\documentclass[11pt,letterpaper]{article}

\pdfoutput=1
\usepackage{fullpage}
\usepackage{makeidx}         
\usepackage{graphicx}        

\usepackage{amsmath,amssymb,amsthm,mathrsfs}
\usepackage[colorlinks,bookmarksopen,bookmarksnumbered,citecolor=red,urlcolor=red]{hyperref}
\usepackage{natbib}
\usepackage[font=footnotesize]{subcaption}
\usepackage[font=footnotesize]{caption}
\usepackage{algpseudocode}
\usepackage{algorithm}
\usepackage{color}
\usepackage{bm}
\usepackage{bbm}
\usepackage{float}
\usepackage{scalerel,stackengine}
\usepackage{epstopdf}
\usepackage{enumitem}
\usepackage[title]{appendix}
\usepackage{accents}
\usepackage{authblk}
\usepackage{mdframed}
\usepackage{tabularx,booktabs}

\usepackage{pdflscape}

\usepackage{psfrag}
\usepackage{xspace}

\newtheoremstyle{nospace}
{4pt}   				
{4pt}   				
{\itshape}  			
{} 		  		    	
{\bfseries} 			
{.}         			
{5pt plus 1pt minus 1pt}
{}          			

\theoremstyle{nospace} \newtheorem{theorem}{Theorem}
\theoremstyle{nospace} 
\theoremstyle{nospace} 
\theoremstyle{nospace} \newtheorem{remark}{Remark}
\theoremstyle{nospace} \newtheorem{example}{Example}
\theoremstyle{nospace} 
\theoremstyle{nospace} 
\theoremstyle{nospace} \newtheorem{assumption}{Assumption}

\newcommand{\argmin}{\operatornamewithlimits{argmin}}

\stackMath
\newcommand\wwidehat[1]{%
\savestack{\tmpbox}{\stretchto{%
  \scaleto{%
    \scalerel*[\widthof{\ensuremath{#1}}]{\kern-.6pt\bigwedge\kern-.6pt}%
    {\rule[-\textheight/2]{1ex}{\textheight}}
  }{\textheight}%
}{0.5ex}}%
\stackon[1pt]{#1}{\tmpbox}%
}

\newmdenv[topline=false,bottomline=false,rightline=false]{leftbox}

\newcommand{\X}{\mathcal{X}}
\newcommand{\U}{\mathcal{U}}
\newcommand{\Y}{\reals^n}
\newcommand{\Hk}{\mathcal{H}_{K}}
\newcommand{\Lin}{\mathcal{L}}
\newcommand{\reals}{\mathbb{R}}
\newcommand{\ip}[2]{\left\langle #1, #2 \right\rangle}

\newcommand{\V}{\mathcal{V}}
\newcommand{\Kb}{K^B}
\newcommand{\Hkb}{\mathcal{H}_K^B}
\newcommand{\Vb}{\mathcal{V}_{B}}
\newcommand{\Vf}{\mathcal{V}_{f}}

\newcommand{\Fl}{\mathcal{F}_{\lambda}}

\newcommand{\bs}{\mathfrak{b}}

\newcommand{\kwp}{\hat{\kappa}}
\newcommand{\kw}{\kappa}
\newcommand{\Hwp}{\mathcal{H}_{\hat{\kappa}}}
\newcommand{\Hw}{\mathcal{H}_{\kappa}}

\newcommand{\Sj}{\mathbb{S}}
\newcommand{\Sjpp}{\mathbb{S}^{>0}}
\newcommand{\Sjp}{\mathbb{S}^{\geq 0}}

\newcommand{\wl}{\underline{w}}
\newcommand{\wu}{\overline{w}}

\newcommand{\xs}{x_i}
\newcommand{\us}{u_i}

\newcommand{\dx}{\delta_x}
\newcommand{\ddx}{\dot{\delta}_x}

\newcommand{\revision}[1]{{\color{black}{#1}}}

\begin{document}

\title{Learning Stabilizable Nonlinear Dynamics \\ with Contraction-Based Regularization}

\author[1]{Sumeet Singh}
\author[1]{Spencer M. Richards}
\author[2]{Vikas Sindhwani}
\author[3]{Jean-Jacques E. Slotine}
\author[1]{\\ Marco Pavone}

\affil[1]{Department of Aeronautics and Astronautics, Stanford University \thanks{ \{ssingh19,spenrich,pavone\}@stanford.edu} }
\affil[2]{Google Brain Robotics, New York \thanks{ sindhwani@google.com } }
\affil[3]{Department of Mechanical Engineering, Massachusetts Institute of Technology \thanks{ jjs@mit.edu }}


\maketitle


\begin{abstract}
We propose a novel  framework for learning stabilizable nonlinear dynamical systems for continuous control tasks in robotics. The key contribution is a control-theoretic regularizer for dynamics fitting rooted in the notion of {\it stabilizability}, a constraint which guarantees the existence of robust tracking controllers for arbitrary open-loop trajectories generated with the learned system. Leveraging tools from contraction theory and statistical learning in Reproducing Kernel Hilbert Spaces, we formulate stabilizable dynamics learning as a functional optimization with convex objective and bi-convex functional constraints. Under a mild structural assumption and relaxation of the functional constraints to sampling-based constraints, we derive the optimal solution with a modified Representer theorem. Finally, we utilize random matrix feature approximations to reduce the dimensionality of the search parameters and formulate an iterative convex optimization algorithm that jointly fits the dynamics functions and searches for a \emph{certificate of stabilizability}. We validate the proposed algorithm in simulation for a planar quadrotor, and on a quadrotor hardware testbed emulating planar dynamics. We verify, both in simulation and on hardware, significantly improved trajectory generation and tracking performance with the control-theoretic regularized model over models learned using traditional regression techniques, especially when learning from small supervised datasets. 
The results support the conjecture that the use of stabilizability constraints as a form of regularization can help prune the hypothesis space in a manner that is tailored to the downstream task of trajectory generation and feedback control, resulting in models that are not only dramatically better conditioned, but also data efficient.
\end{abstract}


\section{Introduction}

The problem of efficiently and accurately estimating an unknown dynamical system, \begin{equation}
    \dot{x}(t) = F(x(t),u(t)), 
\label{ode}
\end{equation} from a small set of sampled trajectories, where $x \in \reals^n$ is the state and $u \in \reals^m$ is the control input, is a central task in model-based Reinforcement Learning (RL). In this setting, a robotic agent strives to pair an estimated  dynamics model with a feedback policy in order to act optimally in a dynamic and uncertain environment. The model of the dynamical system can be continuously updated as the robot experiences the consequences of its actions, and the improved model can be  leveraged for different tasks, affording a natural form of transfer learning. When it works, model-based RL typically offers major improvements in sample efficiency in comparison to state-of-the-art model-free methods such as Policy Gradients~\citep{ChuaCalandraEtAl2018,NagabandiKahnEtAl2017} that do not explicitly estimate the underlying system. Yet, all too often, when standard supervised learning with powerful function approximators such as Deep Neural Networks and Kernel Methods are applied to model complex dynamics, the resulting controllers do not perform on par with model-free RL methods in the limit of increasing sample size, due to compounding errors across long time horizons. The main goal of this paper is to develop a new control-theoretic regularizer for dynamics fitting rooted in the notion of {\it stabilizability}, which guarantees that the existence of a robust tracking controller for arbitrary open-loop trajectories generated with the learned system.

\medskip

\noindent {\bf Problem Statement}: The motion planning task we wish to solve is to compute a (possibly non-stationary) policy mapping state and time to control that drives any given initial state to a desired compact goal region, while satisfying state and control input constraints, and minimizing some task specific performance cost (e.g., control effort and time to completion). However, in this work, we assume that the dynamics function $F(x,u)$ is unknown to us and we are instead provided with a dataset of tuples $\{(\xs, \us, \dot{x}_i)\}_{i=1}^{N}$ taken from a collection of observed trajectories (e.g., expert demonstrations) on the robot. Accordingly, the objective of this work is to learn a dynamics model $F(\cdot,\cdot)$ for the robot that is subsequently amenable for use within standard planning algorithms. 

\medskip

\noindent {\bf Approach Overview}: Our parametrization of the policy takes the form $u^*(t) + k(x^*(t),x(t))$, where $(x^*, u^*)$ is a nominal open-loop state-input control trajectory tuple, and $k(\cdot,\cdot)$ is a feedback tracking controller. The performance of such a policy however, is strongly reliant upon the quality of the computed state-input trajectory and the tracking controller. 

Formally, a reference state-input trajectory tuple $(x^*(t), u^*(t)),\ t \in [0,T]$ for system~\eqref{ode} is termed \emph{exponentially stabilizable at rate $\lambda>0$} if there exists a feedback controller $k : \reals^n \times \reals^n \rightarrow \reals^m$ such that the solution $x(t)$ of the system:
\[
    \dot{x}(t) = F(x(t), u^*(t) + k(x^*(t),x(t))),
\]
converges exponentially to $x^*(t)$ at rate $\lambda$. That is,
\begin{equation}
    \|x(t) - x^*(t)\|_2 \leq C \|x(0) - x^*(0)\|_2 \ e^{-\lambda t}
\label{exp_stab}
\end{equation}
for some constant $C>0$. The \emph{system}~\eqref{ode} is termed \emph{exponentially stabilizable at rate $\lambda$} in an open, connected, bounded region $\X \subset \reals^n$ if all state trajectories $x^*(t)$ satisfying $x^*(t) \in \X,\ \forall t \in [0,T]$ are exponentially stabilizable at rate $\lambda$. 

In this work, we illustrate that na{\"i}ve regression techniques used to estimate the dynamics model from a small set of sample trajectories can yield model estimates that are severely ill-conditioned for trajectory generation and feedback control. Instead, this work advocates for the use of a \emph{constrained} regression approach in which one attempts to solve the following problem:
\begin{align}
    \min_{\hat{F} \in \mathcal{H}} \quad & \sum_{i=1}^{N} \left\| \hat{F}(\xs,\us) - \dot{x}_i \right\|_2^2 + \mu \|\hat{F}\|^2_{\mathcal{H}} \label{prob_gen} \\
    \text{s.t.} \quad & \text{$\hat{F}$ is stabilizable,}
\end{align}
where $\mathcal{H}$ is an appropriate normed function space and $\mu >0$ is a regularization parameter. Note that we use $(\hat{\cdot})$ to differentiate the learned dynamics from the true dynamics. We demonstrate that for systems that are indeed stabilizable, enforcing such a constraint drastically \emph{prunes the hypothesis space}, and therefore plays the role of a \emph{``control-theoretic'' regularizer} that is potentially more powerful and ultimately, more pertinent for the downstream control task of generating and tracking new trajectories.

\medskip

\noindent {\bf Statement of Contributions:} Stabilizability of trajectories is not only a complex task in nonlinear control, but also a difficult notion to capture (in an algebraic sense) within a unified control theory. In this work, we leverage recent advances in contraction theory for control design through the use of \emph{Control Contraction Metrics} (CCMs)~\citep{ManchesterSlotine2017,SinghMajumdarEtAl2017} that turn stabilizability constraints into convex state-dependent Linear Matrix Inequalities (LMIs). Contraction theory~\citep{LohmillerSlotine1998} is a method of analyzing nonlinear systems in a differential framework, i.e., via the associated variational system~\cite[Chp 3]{CrouchSchaft1987}, and is focused on the study of convergence between pairs of state trajectories towards each other. Thus, at its core, contraction explores a stronger notion of stability -- that of incremental stability between solution trajectories, instead of the stability of an equilibrium point or invariant set. Importantly, we harness recent results in~\citep{ManchesterTangEtAl2015,ManchesterSlotine2017,SinghMajumdarEtAl2017} that illustrate how to use contraction theory to obtain a \emph{certificate} for trajectory stabilizability and an accompanying tracking controller with exponential stability properties. For self containment, we provide a brief summary of these results in Section~\ref{sec:ccms}, which in turn will form the foundation of this work.
 
 Our paper makes the following primary contributions. 
 \begin{itemize}
 \item We formulate the learning stabilizable dynamics problem through the lens of control contraction metrics (Section~\ref{sec:prob}). The resulting optimization problem is not only infinite-dimensional, as it is formulated over function spaces, but also infinitely-constrained due to the state-dependent LMI representing the stabilizability constraint. 
 \item Under an arguably weak assumption on the structural form of the true dynamics model and a relaxation of the functional constraints to sampling-based constraints (Section~\ref{sec:reg}), we derive a Representer Theorem~\citep{ScholkoepfSmola2001} specifying the form of the optimal solutions for the dynamics functions and the certificate of stabilizability by leveraging the powerful framework of vector-valued Reproducing Kernel Hilbert Spaces (Section~\ref{sec:deriv}). We motivate the sampling-based relaxation of the functional constraints from a standpoint of viewing the stabilizability condition as a novel control-theoretic \emph{regularizer} for dynamics learning. 
 \item By leveraging theory from randomized matrix feature approximations, we derive a tractable algorithm leveraging alternating convex optimization problems and adaptive sampling to iteratively solve a \emph{finite-dimensional} optimization problem (Section~\ref{sec:soln}). 
 \item We perform an extensive set of numerical simulations on a 6-state, 2-input planar quadrotor model and provide a comprehensive study of various aspects of the iterative algorithm. Specifically, we demonstrate that na{\"i}ve regression-based dynamics learning can yield estimated models that generate completely unstabilizable trajectories. In contrast, the control-theoretic regularized model generates vastly superior quality trackable trajectories, especially when learning from small supervised datasets (Sections~\ref{sec:pvtol_sim} and~\ref{sec:results}). 
 \item We validate our algorithm on a quadrotor testbed (Section~\ref{sec:pvtol_exp}) with partially closed control loops to emulate a planar quadrotor, where we verify that the stabilizability regularization effects in low-data regimes observed in simulations does indeed generalize to real-world noisy data. In particular, with just 150 noisy tuples of $(x,u,\dot{x})$, we are able to stably track a challenging test trajectory, which is generated with the learned model and substantially different from any of the training data. In contrast, a model learned using traditional regression techniques leads to consistently unstable behavior and eventual failure as the quadrotor repeatedly flips out of control and crashes (see Figure~\ref{fig:150_overlays}).
 \end{itemize}
 \begin{figure}[H]
\centering
\begin{subfigure}[t]{0.49\textwidth}
	\includegraphics[width=1\textwidth,height=6cm]{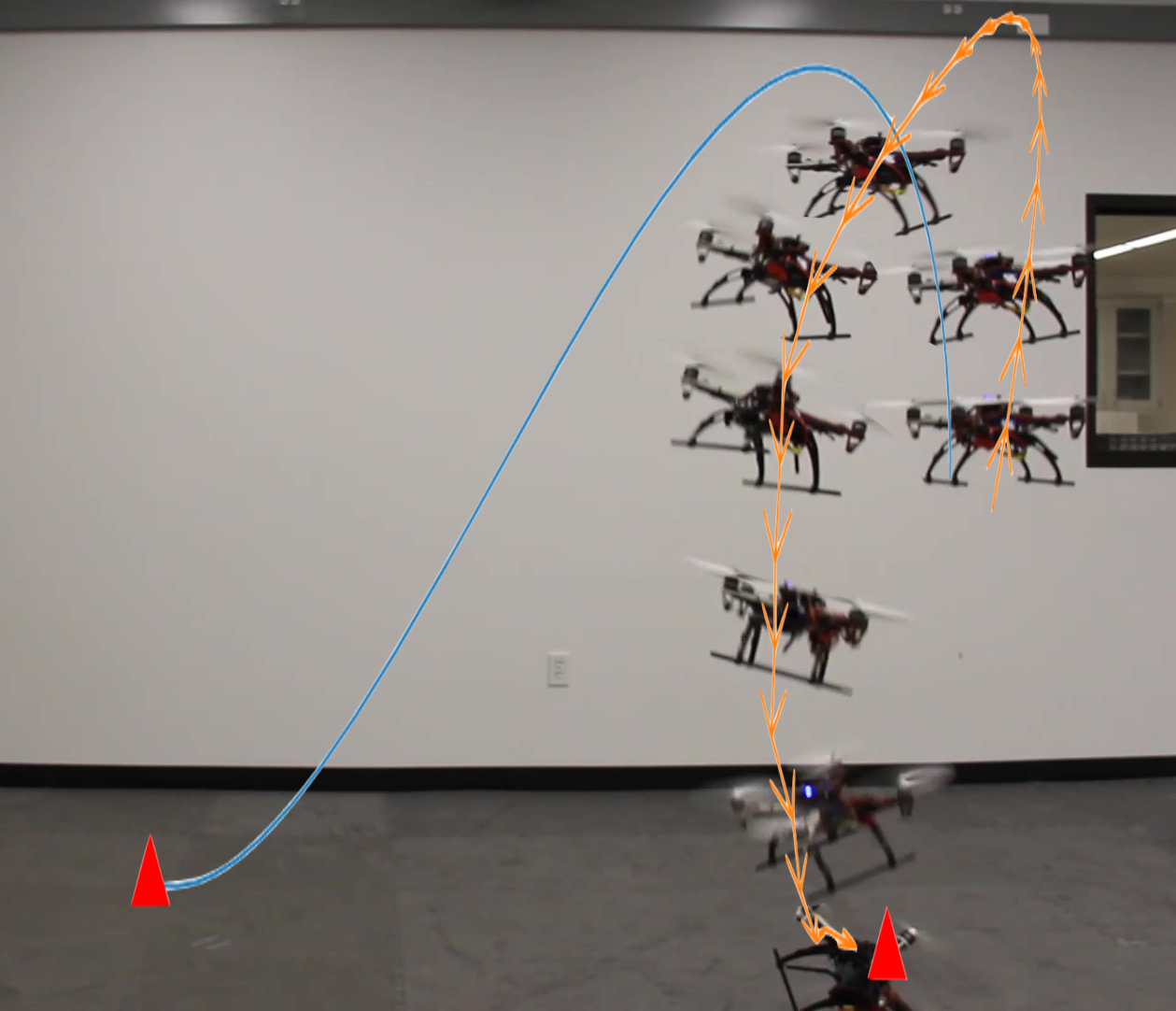}
	\label{fig:l2_overlay}
\end{subfigure} \vspace{-3mm}
\begin{subfigure}[t]{0.49\textwidth}
	\includegraphics[width=1\textwidth,height=6cm]{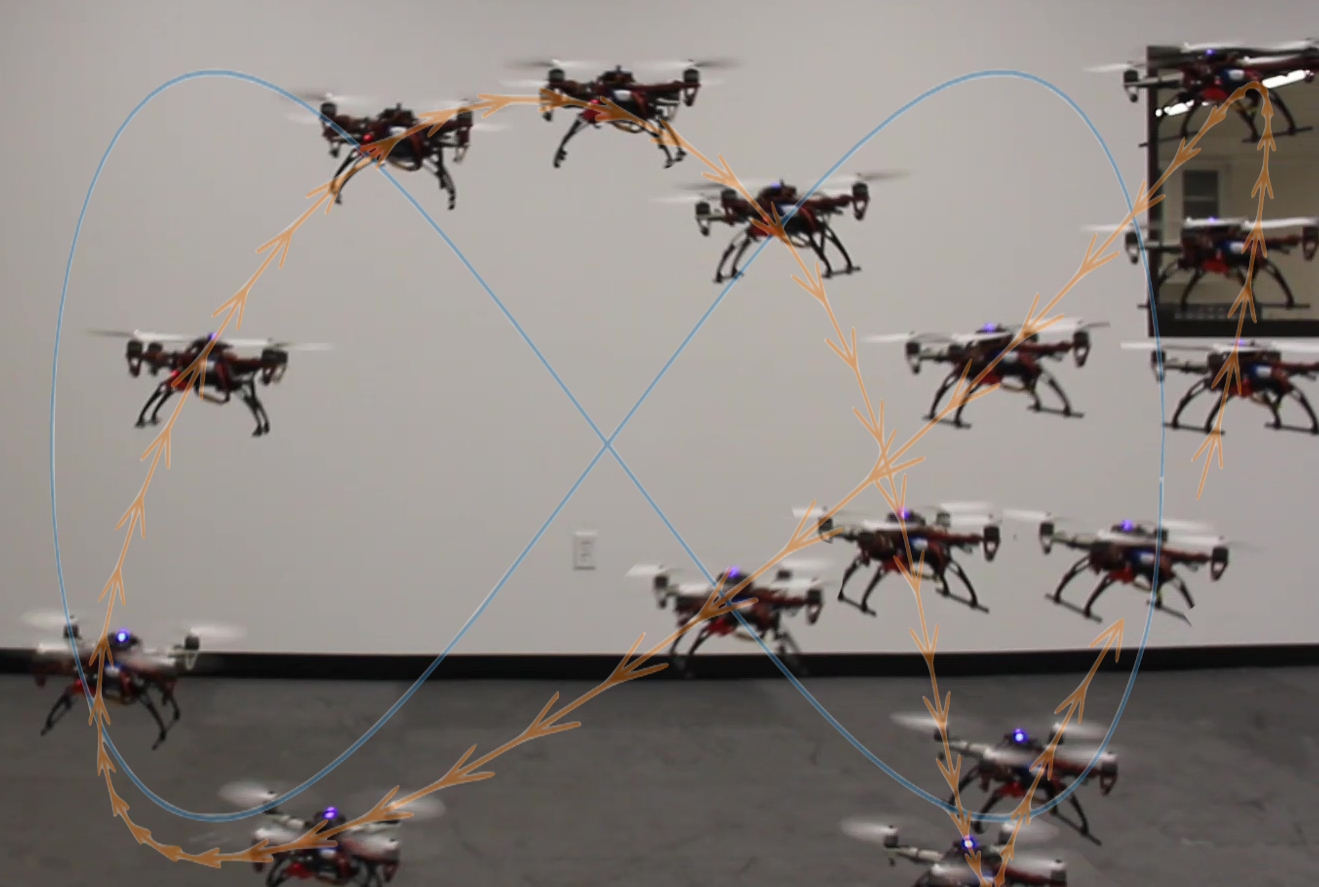}
	\label{fig:ccm_overlay}
\end{subfigure}
	\caption{Time-lapse of a quadrotor trying to execute a figure-eight maneuver (blue curve) using a reference trajectory and an LQR feedback tracking controller generated using the learned dynamical system. \emph{Left}: Model learned using traditional ridge-regression; \emph{Right}: Model learned using control-theoretic regularization proposed within this work. The models were trained with the same, extremely limited (150 points) set of $(x,u,\dot{x})$ supervisory tuples. The quadrotor consistently failed and crashed into the floor with the trajectory and controller generated by the model learned with ridge-regression; the red triangles mark the points along the reference and actual trajectories at moment of crash -- a separation of 1.6 m. In contrast, despite imperfect tracking (not unexpected given the extremely limited amount of supervision given to the learning algorithm), which leads to a slight graze along the floor at one point during the maneuver, the quadrotor manages to maintain bounded tracking error while using the model learned with control-theoretic regularization.}
	\label{fig:150_overlays} 
\end{figure}
 
 A preliminary version of this paper was presented at WAFR 2018~\citep{SinghSindhwaniEtAl2018}. In this revised and extended version, we include the following additional contributions: (i) rigorous derivation of the stabilizability-regularized finite-dimensional optimization problem using RKHS theory and random matrix features; (ii) extensive additional numerical studies into the convergence behavior of the iterative algorithm and comparison with traditional ridge-regression techniques; and (iii) validation of the algorithm on a quadrotor testbed with partially closed control loops to emulate a planar quadrotor. 

\medskip

\noindent {\bf Related Work}: Model-based RL has enjoyed considerable success in various application domains within robotics such as underwater vehicles~\citep{CuiYangEtAl2017}, soft robotic manipulators~\citep{ThuruthelFaloticoEtAl2019}, and control of agents with non-stationary dynamics~\citep{OhnishiWangEtAl2019}. While the literature on model-based RL is substantial; see~\citep{PolydorosNalpantidis2017} for a recent review, we focus our attention on five broad categories relevant to the problem we address in this work. Namely, these are: (i) direct regression for learning the full dynamics, where one ignores any control-theoretic notions tied to the learning task and treats dynamics estimation as a standard regression problem; (ii) residual learning, where one only attempts to learn \emph{corrections} to a nominal prediction model that may have been derived, for example, from physics-based reasoning; (iii) uncertainty-aware model-based RL, where one tries to additionally represent the uncertainty in the learned model using probabilistic representations that are subsequently leveraged within the planning phase using robust or stochastic control techniques; (iv) hybrid model-based/model-free methods; and (v) imitation learning, where one learns dynamical representations of stable closed-loop behavior for a set of outputs (e.g., the end-effector on a robotic arm), and assumes knowledge of the robot controlled dynamics to realize the learned closed-loop motion, for instance, using dynamic inversion.  

The simplest approach to learning dynamics is to ignore stabilizability and treat the problem as a standard one-step time series regression task~\citep{PunjaniAbbeel2015,BansalAkametaluEtAl2016,NagabandiKahnEtAl2017,PolydorosNalpantidis2017}. However, coarse dynamics models trained on limited training data typically generate trajectories that rapidly diverge from expected paths, inducing controllers that are ineffective when applied to the true system. This divergence can be reduced by expanding the training data with corrections to boost multi-step prediction accuracy~\citep{VenkatramanHebertEtAl2015, VenkatramanCapobiancoEtAl2016}. Despite being effective, these methods are still heuristic in the sense that the existence of a stabilizing feedback controller is not explicitly guaranteed. Alternatively, one can leverage strong physics-based priors and use learning to only regress the unmodeled dynamics. For instance,~\citep{MohajerinMozifianEtAl2019,ShiShiEtAl2019,PunjaniAbbeel2015} aim to capture the unmodeled aerodynamic disturbance terms as corrections to a prior rigid body dynamics model.~\citep{PunjaniAbbeel2015} accomplish this for helicopter dynamics using a deep neural network, but then do not use the learned model for control.~\citep{ShiShiEtAl2019} attempt to capture the unmodeled ground-effect forces on quadrotors to build better controllers for near-ground tracking and precision landing.~\citep{MohajerinMozifianEtAl2019} leverage a residual RNN in combination with a rigid-body model to generate time-series predictions for linear and angular velocities of a quadrotor as a function of current state and candidate future motor inputs, but do not use the model for closed-loop control. Finally,~\citep{ZhouHelwaEtAl2017} adopt a different perspective to learning ``corrections" in that they attempt to learn the inverse dynamics (output to reference) for a system and pre-cascade the resulting predictions to correct an existing controller's reference signal in order to improve trajectory tracking performance. The approach relies on the existence of a stabilizing controller and the stability of the system's zero dynamics, thereby decoupling the effects of learning from stability. In similar spirit,~\citep{TaylorDorobantuEtAl2019} leverage input-output feedback linearization to derive a Control Lyapunov Function (CLF) for a nominal dynamics model, assume that this function is a CLF for the actual dynamics as well, and regress only the correction terms in the derivative of this CLF. While leveraging physics-based priors can certainly be powerful, especially when the residual errors to be learned are small enough such that the system is feedback stabilizable with a controller derived from the physics model, in this work we are interested in the far more challenging scenario when such priors are unavailable and the full dynamics model must be learned from scratch. While exemplified using quadrotor models that can certainly be accurately stabilized even in the absence of learning, the insights provided in this work shed light on fundamental topics in the context of control-theoretic learning, which hopefully may influence dynamics-learning methods in more complex settings where priors are unavailable or too simple to be useful for adequate control. 

An alternative strategy to cope with error in the learned dynamics model is to use uncertainty-aware model-based RL where control policies are optimized with respect to stochastic rollouts from probabilistic dynamics models~\citep{KocijanMurray-SmithEtAl2004,KamtheDeisenroth2018,DeisenrothRasmussen2011,ChuaCalandraEtAl2018}. For instance, PILCO~\citep{DeisenrothRasmussen2011} leverages a Gaussian Process (GP) state transition model and moment matching to analytically estimate the expected cost of a rollout with respect to the induced distribution.~\citep{KamtheDeisenroth2018} extend this formulation using nonlinear model predictive control (MPC) to incorporate chance constraints.~\citep{ChuaCalandraEtAl2018} leverage an ensemble of probabilistic models to capture both epistemic (i.e., model) and aleatoric (i.e., intrinsic) uncertainty, and compute their control policy in receding horizon fashion through finite sample approximation of the random cost. Probabilistic models such as GPs may also be used to capture the residual error between a nominal physics-based model and the true dynamics. In~\citep{OstafewSchoelligEtAl2016}, a GP is incrementally learned over multiple trials to capture unmodeled disturbances. The 3$\sigma$ prediction range is subsequently leveraged to formulate chance constraints as a robust nonlinear MPC problem. The goal of~\citep{FisacAkametaluEtAl2017} and~\citep{BerkenkampTurchettaEtAl2017} is motivated from a safety perspective, where one wishes to actively learn a control policy while remaining ``safe" in the presence of unmodeled dynamics, represented as GPs. The authors in~\citep{FisacAkametaluEtAl2017} leverage Hamilton-Jacobi reachability analysis to give high-probability invariance guarantees for a region of the state-space within which the learning controller is free to explore. On the other hand, ~\citep{BerkenkampTurchettaEtAl2017} utilize Lyapunov analysis and smoothness arguments to incrementally grow the Lyapunov function's region of attraction while simultaneously updating the GP. For the special case where the underlying dynamics are linear-time-invariant,~\citep{DeanTuEtAl2019} derive high-probability convergence rates for the estimated model and leverage system-level robust control techniques~\citep{WangMatniEtAl2019} for guaranteeing state and control constraint satisfaction. 

While utilizing probabilistic prediction models along with a control strategy that incorporates this uncertainty, such as robust or approximate stochastic MPC, can certainly help guard against imperfect dynamics models, large uncertainty in the dynamics can lead to overly conservative strategies. This is true especially when the learned model is not merely a correction or residual term, or if the probabilistic model is computationally intractable to use within planning (e.g., GPs without additional sparsifying simplifications), thereby forcing conservative approximations. Finally, with the exception of the ``safe" RL methods mentioned above, the learning algorithms themselves do not incorporate knowledge of the downstream application of the function being regressed, in that learning is viewed purely from a \emph{statistical point-of-view, rather than within a control-theoretic context.} 

More recently, hybrid combinations of model-based and model-free techniques have gained attention within the learning community. The authors in~\citep{BansalCalandraEtAl2017} use Bayesian optimization to find an optimal linear dynamics model whose induced MPC policy minimizes the task-specific cost. In similar spirit,~\citep{AmosRodriguezEtAl2018} differentiate through the fixed-point solutions of a parametric MPC problem to find optimal MPC cost and dynamics functions in order to minimize the actual task-specific cost.~\citep{NagabandiKahnEtAl2017} use behavioral cloning with respect to an MPC policy generated from a learned dynamics model to initialize model-free policy fine-tuning. The works in~\citep{LevineFinnEtAl2016,FinnLevineEtAl2016,ChebotarHausmanEtAl2017} leverage subroutines where local time-varying dynamics are fitted around a set of policy rollouts, and then used to perform trajectory optimization via an LQR backward pass. The induced local linear-time-varying policy from this rollout is then used as a supervisory signal for global policy optimization. While these lines of work try to frame dynamics fitting within the downstream context of the task, thereby imbuing the resulting learning algorithm with a more closed-loop flavor, the learned dynamics may be substantially different from the actual dynamics of the robot since, with the exception of the local time-varying dynamics fitting, the true goal is to optimize the task-specific cost. This can yield distorted dynamic models whose induced policies are more cost-optimal than policies extracted from the true dynamics. Thus, while the work presented herein espouses a closed-loop learning ideology, it does so from the control-theoretic perspective of trajectory stabilizability, i.e., the true objective is dynamics fitting which will subsequently be used to derive optimal trajectories and tracking controllers. 

Finally, we address lines of work closest in spirit to this work. Learning dynamical systems satisfying some desirable stability properties (such as asymptotic stability about an equilibrium point, e.g., for point-to-point motion) has been studied in the autonomous case, $\dot{x}(t) = f(x(t))$, in the context of imitation learning. In this line of work, one assumes perfect knowledge and invertibility of the robot's \emph{controlled} dynamics to solve for the input that realizes this desirable closed-loop motion~\citep{LemmeNeumannEtAl2014,Khansari-ZadehKhatib2017,RavichandarSalehiEtAl2017,Khansari-ZadehBillard2011,MedinaBillard2017}. In particular, for a vector-valued RKHS formulation in the autonomous case with constant (identity) contraction metric, see~\citep{SindhwaniTuEtAl2018}. Crucially, in our work, we \emph{do not} require knowledge or invertibility of the robot's controlled dynamics. We seek to learn the full controlled dynamics of the robot, under the constraint that the resulting learned dynamics generate dynamically feasible and most importantly, stabilizable trajectories. Thus, this work generalizes existing literature by additionally incorporating the controllability limitations of the robot within the learning problem. 

The tools we develop may also be used to extend standard adaptive robot control design, such as~\citep{SlotineLi1987} -- a technique which achieves stable concurrent learning and control using a combination of physical basis functions and general mathematical expansions, e.g. radial basis function approximations~\citep{SannerSlotine1992}. Notably, our work allows us to handle complex underactuated systems -- a consequence of the significantly more powerful function approximation framework developed herein, as well as of the use of a differential (rather than classical) Lyapunov-like setting, as we shall detail.

\medskip

\noindent {\bf Notation}: Let $\Sj_j$ be the set of symmetric matrices in $\reals^{j\times j}$ and denote $\Sjp_j$, respectively $\Sjpp_j$, to be the set of symmetric positive semi-definite, respectively, positive definite matrices in $\reals^{j\times j}$. Given a matrix $X$, let $\widehat{X}:=  X + X^T$. 
 We denote the components of a vector $y\in \reals^n$ as $y^j,\ j=1,\ldots,n,$ its Euclidean norm as $\|y\|$, and its weighted norm as $\|y\|_A$ where $\|y\|_A = \sqrt{y^T A y}$ for some $A \in \Sjpp_n$. Let $\partial_{y} G(x)$ denote a matrix with $(i,j)^{\text{th}}$ entry given by the Lie derivative of the function $G_{ij}(x)$ along the vector $y$. Finally, let $\bar{\lambda}(A)$ and $\underline{\lambda}(A)$ denote the maximum and minimum eigenvalues of a square matrix $A$. 

\section{Problem Formulation and Solution Methodology} \label{sec:prob_set}

In this section we formally outline the structure of the problem we wish to solve and describe a general solution methodology rooted in model-based RL. To motivate the contributions of this work, we additionally present an attempt at a  solution that uses traditional model-fitting techniques, and demonstrate how it fails to capture the nuances of the problem and ultimately yields sub-par results. 

Consider a robotic system with state $x \in \X$, where $\X$ is an open, connected, bounded subset of $\reals^n$, and control $u \in \U$, where $\U$ is a closed, bounded subset of $\reals^m$, governed by the following continuous-time dynamical system:
\[
	\dot{x}(t) = F(x(t), u(t)),
\]
where $F$ is Lipschitz continuous in the state for fixed control, so that for any measurable control function $u(\cdot)$, there exists a unique state trajectory. The motion planning \emph{task} we wish to solve is to find a (possibly non-stationary) policy $\pi: \X \times \reals \rightarrow \U$ that (i) drives the state $x$ to a compact region $\X_{\text{goal}}\subseteq \X$, (ii) satisfies the state and input constraints, and (iii) minimizes a quadratic cost:
\[
J(\pi) := \int_{0}^{T_{\mathrm{goal}}} 1 + \|\pi(x(t), t)\|_R^2 \ dt,
\]
where $R \in \Sjpp_m$ and $T_{\mathrm{goal}}$ is the first time $x(t)$ enters $\X_{\text{goal}}$. While there exist several methods in the literature on how to solve this problem given knowledge of the dynamical system, in this work, we assume that we do not know the governing model $F(x,u)$. The \emph{problem} we wish to address is how to solve the above motion planning task, given a dataset of tuples $\{ (\xs,\us,\dot{x}_i) \}_{i=1}^{N}$ extracted from observed trajectories on the robot. 

The solution approach presented in this work adopts the model-based RL paradigm, whereby one first estimates a model of the dynamical system $\hat{F}(x,u)$ using some form of regression, and then uses the learned model to solve the motion planning task with traditional planning algorithms. In this work, our strategy to solve the planning task is to parameterize general state-feedback policies as a sum of a nominal (open-loop) input $u^{*}$ and a feedback term designed to track the nominal state trajectory $x^{*}$ (induced by $u^{*}$):
\begin{equation}
	\pi(x(t),\, t) = u^{*}(t) + k(x^{*}(t), x(t)).
\label{net_cont}
\end{equation}
This formulation represents a compromise between the general class of state-feedback control laws (a computationally intractable space over which to optimize) and a purely open-loop formulation (i.e., no tracking). Note that we do \emph{not} present a new methodology for solving the planning task. Specifically, it is assumed that there exists an algorithm for computing (i) the open-loop state and control trajectories $(x^*(t), u^*(t))$ that minimize the open-loop cost:
\[
	J(u(\cdot)) = \int_{0}^{T_{\mathrm{goal}}} 1 + \|u(t)\|_R^2 \ dt,
\]
and (ii) the feedback tracking controller $k(\cdot,\cdot)$, given a dynamical model. The focus of this paper is on how to design the regression algorithm for computing the model estimate $\hat{F}$. 

\subsection{Motivating Example}\label{sec:pvtol_sim}

We ground the formalism within the following running example that will feature throughout this work. 

\begin{example}[PVTOL]
Consider the 6-state planar vertical-takeoff-vertical-landing (PVTOL) system depicted in Figure~\ref{fig:PVTOL_define}. The system is defined by the state $(p_x,p_z,\phi,v_x,v_z,\dot{\phi})$, where $(p_x,p_z)$ is the position in the 2D plane, $(v_x,v_z)$ is the body-reference velocity, and $(\phi,\dot{\phi})$ are the roll and angular rate respectively, and $u \in \reals^2_{>0}$ are the controlled motor thrusts. The true dynamics are given by:
\begin{equation}
    \dot{x}(t) = \begin{bmatrix} v_x \cos\phi - v_z \sin\phi \\ v_x\sin\phi + v_z\cos\phi \\ \dot{\phi} \\ v_z\dot{\phi} - g\sin\phi \\ -v_x\dot{\phi} - g\cos\phi \\ 0 \end{bmatrix} + \begin{bmatrix} 0&0\\0&0 \\0&0 \\0&0 \\ (1/m) &(1/m) \\ l/J & (-l/J) \end{bmatrix}u,
\label{nom_pvtol_dyn}
\end{equation}
where $g$ is the acceleration due to gravity, $m$ is the mass, $l$ is the moment-arm of the thrusters, and $J$ is the moment of inertia about the roll axis. 
\end{example}
\begin{figure}[h]
\centering
	\includegraphics[width=0.5\textwidth]{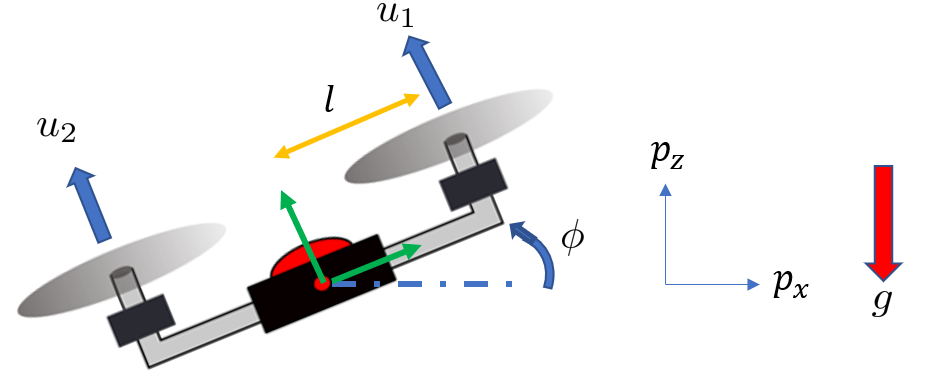}
	\caption{Definition of planar quadrotor state variables: $l$ denotes the thrust moment arm (symmetric), and $u_1$ and $u_2$ denote the right and left thrust forces respectively.}
	\label{fig:PVTOL_define} 
\end{figure}
The planar quadrotor is a complex non-minimum phase dynamical system that has been heavily featured within the acrobatic robotics literature and therefore serves as a suitable case-study. 

\subsubsection{Solution Parametrization}

The dynamics assume the general control-affine form:
\begin{equation}
	\dot{x}(t) = f(x(t)) + B(x(t)) u(t),
\label{cntrl_affine}
\end{equation}
where $f : \X \rightarrow \reals^n$, and $B: \X \rightarrow \reals^{n\times m}$ is the input matrix, depicted in column-stacked form as $(b_1,\ldots,b_m)$. Let us define the model estimate also in control-affine form as $\dot{x} = \hat{f}(x) + \hat{B}(x) u$, where $\hat{B} = (\hat{b}_1,\ldots,\hat{b}_m)$. Consider, as a first solution attempt, the following linear parametrization for the vector-valued functions $\hat{f}$ and $\hat{b}_j$:
\begin{subequations} \label{dyn_param}
\begin{align}
    \hat{f}(x) &= \Phi_f(x)^T \alpha, \\
    \hat{b}_j(x) &= \Phi_b(x)^T \beta_j  \quad j \in \{1,\ldots, m\}, 
\end{align}
\end{subequations}
where $\alpha \in \reals^{d_f}$, $\beta_j \in \reals^{d_b}$ are constant vectors to be optimized over, and $\Phi_f : \X \rightarrow \reals^{d_f\times n}$, $\Phi_b : \X \rightarrow \reals^{d_b \times n}$ are a priori chosen feature mappings. To replicate the sparsity structure of the PVTOL input matrix, the feature matrix $\Phi_b$ has all zeros in its first $n-m$ columns. 

The justification for a linear model and the construction of the feature mappings will be elaborated upon later. At this moment, we wish to study the quality of the learned models obtained from solving the following convex optimization problem:
\begin{equation}
	\min_{\alpha, \{\beta_j\}} \qquad \sum_{i=1}^{N}  \| \hat{f}(\xs) + \hat{B}(\xs)\us - \dot{x}_i \|^2 + \mu_f \|\alpha\|^2 + \mu_b \sum_{j=1}^{m} \|\beta_j\|^2, 
\label{rr_opt}
\end{equation}
where $\mu_f, \mu_b > 0$ are given regularization constants. Note that the above optimization corresponds to the ubiquitous ridge-regression problem and is therefore a viable solution approach. 

To evaluate the feasibility of this solution approach, we extracted a collection of training tuples $ \{(\xs,\us,\dot{x}_i)\}_{i=1}^{N}$ from simulations of the PVTOL system without any noise (for further details, please see Section~\ref{sec:results}). We learned three models: (i) {\bf N-R}: un-regularized model\footnote{A small penalty on $\mu_b$ was necessary since the input feature matrix $\Phi_b$ was chosen to be constant (and degenerate) to reflect the fact that the input matrix for the PVTOL system is constant.} ($\mu_f = 0$, $\mu_b = 10^{-6}$), (ii) {\bf R-R}: standard ridge-regularized model with $\mu_f = 10^{-4}, \mu_b = 10^{-6}$, and (iii) {\bf CCM-R}: control-theoretic regularized model, corresponding to the algorithm proposed within this work and elaborated upon in the remaining paper. 

We learned four versions of the model corresponding to varying training dataset sizes with $N \in \{100, 250, 500, 1000\}$. The dimensions of $\alpha$ and $\beta_j$ were both 576 (corresponding to 96 parameters per state dimension). The feature mappings themselves are described in Section~\ref{sec:results} and Appendix~\ref{app:prob_params}. The regularization constants were held fixed for all $N$. 

\subsubsection{Evaluation}

The evaluation corresponded to the motion planning task of generating and tracking trajectories using the learned models. We gridded the $(p_x,p_z)$ plane to create a set of 120 initial conditions between 4 m and 12 m away from $(0,0)$, and randomly sampled the other states for the rest of the initial conditions. These conditions were \emph{held fixed} for all models and for all training dataset sizes to evaluate model improvement.

For each model at each value of $N$, the evaluation task was to (i) solve a trajectory optimization problem to compute a dynamically feasible trajectory for the learned model to go from initial state $x_0$ to the goal state -- a stable hover at $(0,0)$ at near-zero velocity; and (ii) track this trajectory with a feedback controller computed using time-varying LQR (TV-LQR). Note that all simulations without \revision{any feedback controller (i.e., open-loop control rollouts) led to the PVTOL crashing}. This is understandable since the dynamics fitting objective does not optimize for \emph{multi-step} error. The trajectory optimization step was solved as a fixed-endpoint, fixed-final time optimal control problem using the Chebyshev pseudospectral method~\citep{FahrooRoss2002} with the objective of minimizing $\int_{0}^T \|u(t)\|^2_R\  dt$. The final time $T$ for a given initial condition was held fixed between all models. Note that 120 trajectory optimization problems were solved for each model and each value of $N$.

Figure~\ref{fig:box_all} shows a boxplot comparison of the trajectory-wise RMS full state errors ($\|x(t)-x^*(t)\|$ where $x^*(t)$ is the reference trajectory obtained from the optimizer and $x(t)$ is the actual realized trajectory) for each model and all training dataset sizes. As $N$ increases, the spread of the RMS errors decreases for both R-R and CCM-R models as expected. However, we see that the N-R model generates \emph{several} unstable trajectories for $N \in \{100, 500, 1000\}$, indicating the need for a form of regularization.

For $N=100$ (which is at the extreme lower limit of the necessary number of samples since there are $96$ features for each dimension of the dynamics function), both N-R and R-R models generate a large number of unstable trajectories. In contrast, \emph{all} trajectories generated with the CCM-R model were successfully tracked with bounded error.  The CCM-R model consistently achieves a lower RMS error distribution than both the N-R and R-R models \emph{for all training dataset sizes}. Most notable, however, is its performance when the number of supervision training samples is small (i.e., $N \in \{100, 250\}$) and there is considerable risk of overfitting. It appears the stabilizability constraints leveraged to compute the CCM-R model have a notable regularizing effect on the resulting model trajectories (recall that the initial conditions of the trajectories are held fixed between the models). In Figure~\ref{fig:unstable}, we highlight two trajectories that start from the \emph{same initial conditions} -- one generated and tracked using the R-R model, the other using the CCM-R model, for $N=250$. Overlaid on the plot are snapshots of the vehicle outline itself, illustrating the aggressive flight-regime of the trajectories (the initial bank angle is $40^\mathrm{o}$). While tracking the R-R model generated trajectory eventually ends in \revision{complete loss of control}, the system successfully tracks the CCM-R model generated trajectory to the stable hover at $(0,0$).

\begin{figure}[h]
    \centering
    \includegraphics[width=0.9\textwidth,clip]{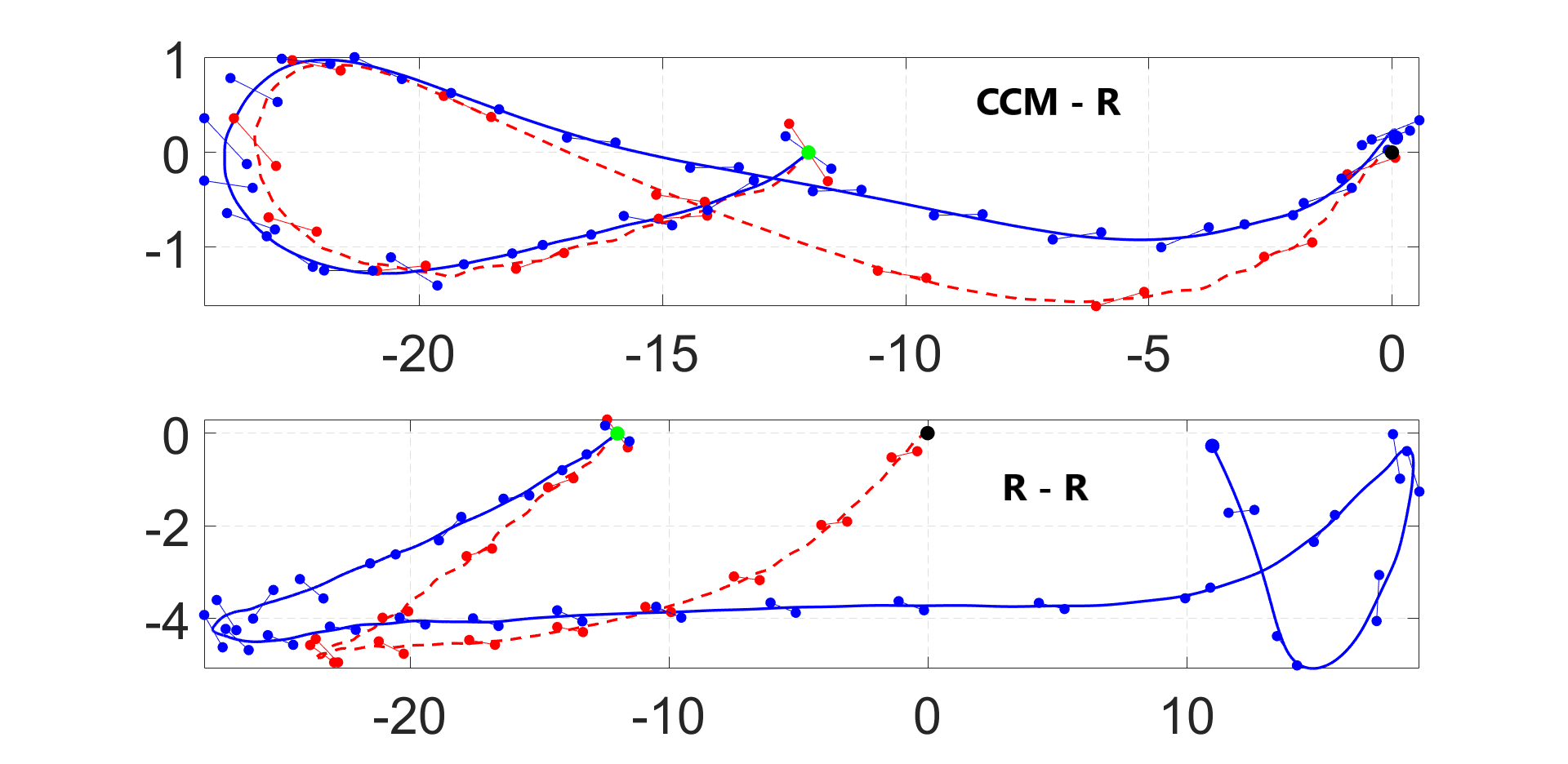}
    \caption{Comparison of reference and tracked trajectories in the $(p_x,p_z)$ plane for R-R and CCM-R models starting at the same initial conditions with $N=250$. Red (dashed): nominal, Blue (solid): actual, Green dot: start, Black dot: nominal endpoint, blue dot: actual endpoint; \emph{Top:} CCM-R, \emph{Bottom:} R-R. The vehicle successfully tracks the CCM-R model generated trajectory to the stable hover at $(0,0)$ while losing control when attempting to track the R-R model generated trajectory.}
        \label{fig:unstable}
\end{figure}

\begin{landscape}
\topskip0pt
\vspace*{\fill}
\begin{figure}[h]
    \centering
    \includegraphics[width=1.35\textwidth,clip]{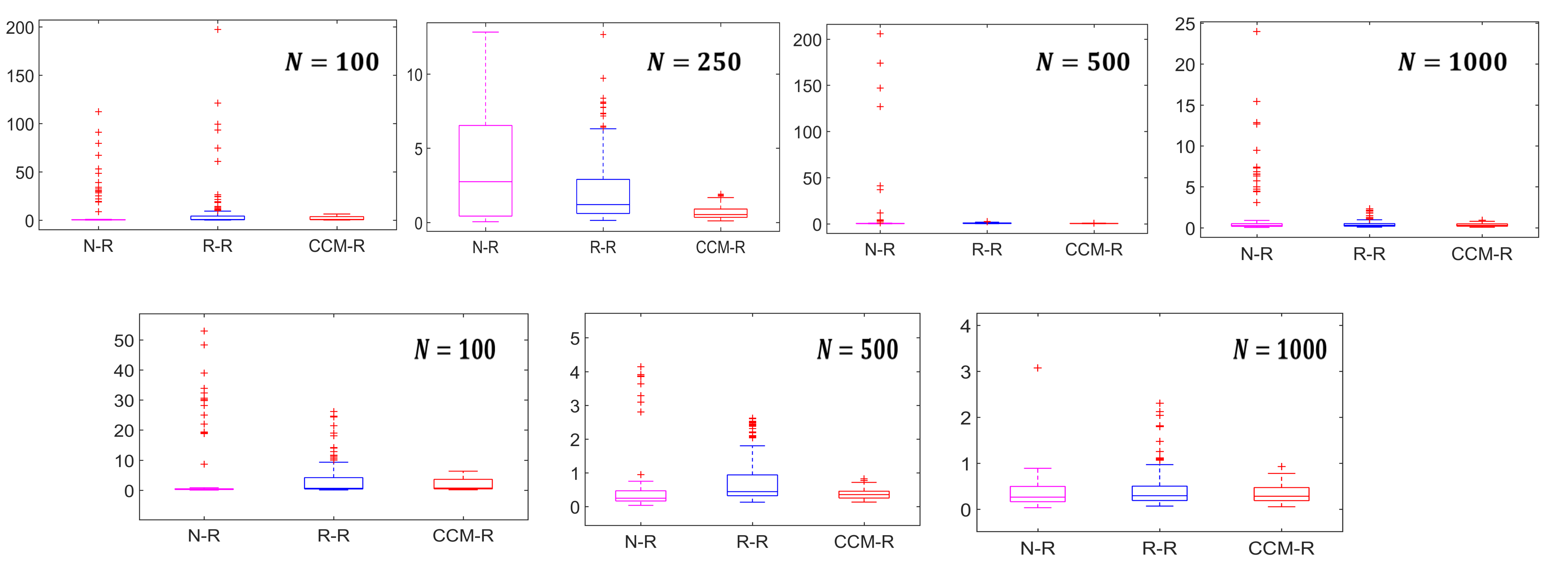}
    \caption{Box-whisker plot comparison of trajectory-wise RMS state-tracking errors over all 120 trajectories for each model and all training dataset sizes. \emph{Top row, left-to-right:} $N=100, 250, 500, 1000$; \emph{Bottom row, left-to-right:} $N=100, 500, 1000$ (zoomed in). The box edges correspond to the $25$th, median, and $75$th percentiles; the whiskers extend beyond the box for an additional 1.5 times the interquartile range; outliers, classified as trajectories with RMS errors past this range, are marked with red crosses. Notice the presence of unstable trajectories for N-R at all values of $N$ and for R-R at $N=100, 250$. The CCM-R model dominates the other two \emph{at all values of $N$}, particularly for $N \in \{ 100, 250 \}$. }
        \label{fig:box_all}
\end{figure}
\vspace*{\fill}
\end{landscape}

\subsubsection{Effect of Regularization}

At this point, one might wonder if the choice of the regularization parameter $\mu_f$ may be sub-optimal for the R-R model. Traditionally, such parameters are tuned using regression error on a validation dataset. In Figure~\ref{fig:mu_vs_reg} we plot the mean regression error $\|\hat{f} + \hat{B}u - \dot{x} \|$ over an independently sampled validation dataset of 2000 demonstration tuples, as a function of the regularization parameter $\mu_f$, for all  R-R models.
\begin{figure}[h]
    \centering
    \includegraphics[width=0.8\textwidth,clip]{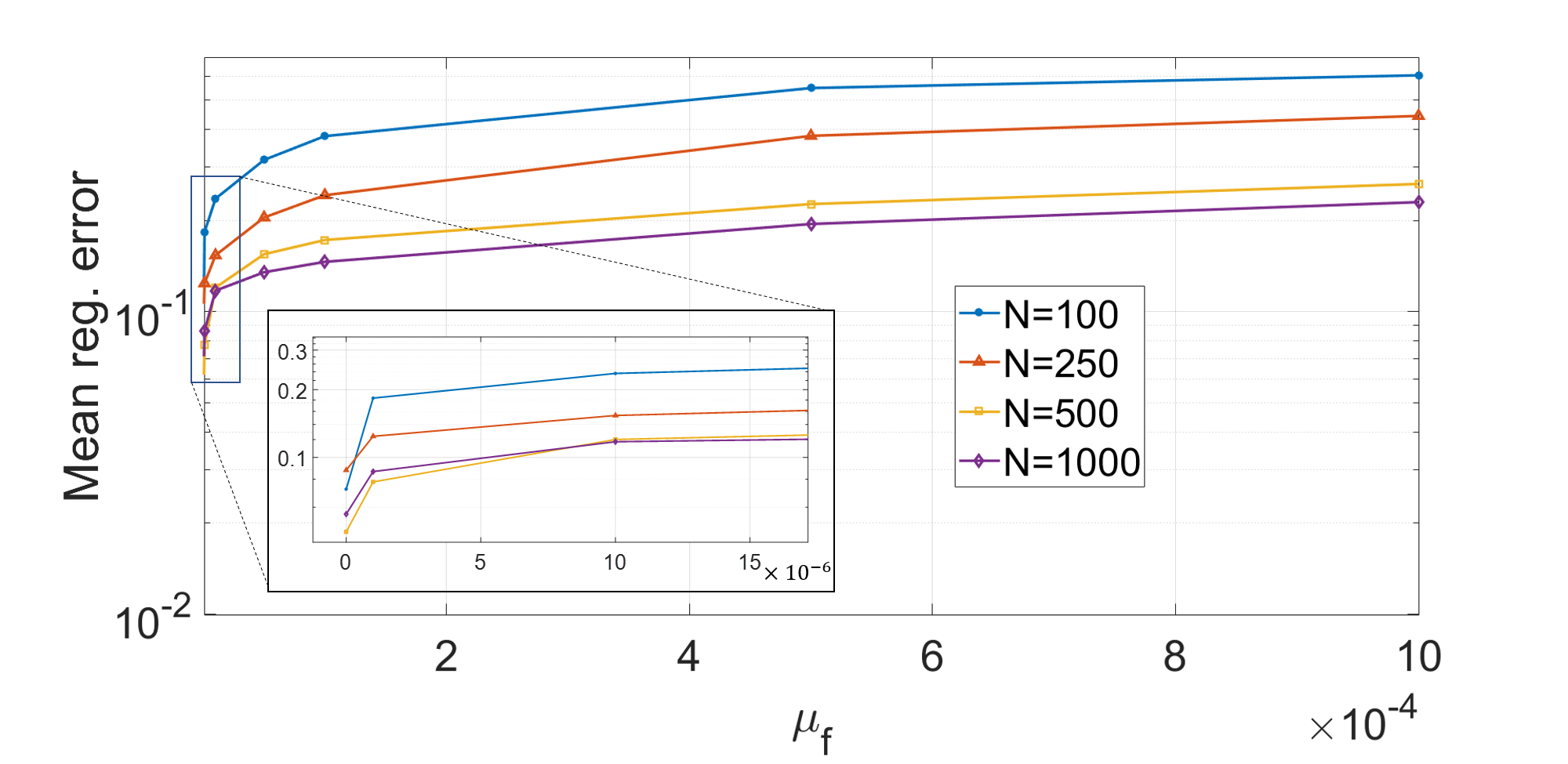}
    \caption{Mean regression error over an independent validation dataset as a function of $\mu_f$ for the RR model learned using~\eqref{rr_opt}, with varying training set size $N$. The best out of sample performance is achieved with $\mu_f = 0.0$. The constant $\mu_b$ is fixed at $10^{-6}$. }
        \label{fig:mu_vs_reg}
\end{figure}

The plot illustrates that the best out-of-sample performance is achieved with $\mu_f = 0.0$. However, this corresponds to the N-R model which, as we learned in the previous section, generated several unstable trajectories for all training dataset sizes. This is not a surprising result; the feature mapping used as the basis for the dynamics model corresponds to the randomized matrix approximation of a reproducing kernel (see Section~\ref{sec:feat}). Recent results, as in~\citep{LiangRakhlin2018} and references within, corroborate such a pattern and even advocate for ``ridgeless'' regression. 

Given that the CCM-R model uses the \emph{same} feature mapping as the R-R and N-R models (i.e., the model capacity of all three models is the same), and is given the same set of demonstration tuples, it appears that traditional model-fitting techniques such as ridge-regression and associated hyper-parameter tuning rules are ill-suited to learn representations of dynamics that are appropriate for planning and control. This motivates the need for \emph{constrained} dynamics learning, where the notion of \emph{model stabilizability} is encoded as a constraint within the learning algorithm (as opposed to the unconstrained optimization in~\eqref{rr_opt}). In the next section, we introduce conditions for nonlinear trajectory stabilizability which can be encoded as algebraic constraints within a model learning algorithm to \emph{prune the hypothesis space in a manner that is tailored to the downstream task of trajectory generation and feedback control.}

\section{Review of Contraction Theory} \label{sec:ccms}

The core principle behind contraction theory~\citep{LohmillerSlotine1998} is to study the evolution of distance between any two \emph{arbitrarily close} neighboring trajectories and draw conclusions upon the distance between \emph{any finitely apart} pair of trajectories.  Given an autonomous system of the form: $\dot{x}(t) = f(x(t))$, consider two neighboring trajectories separated by an infinitesimal (virtual) displacement $\delta_x$; formally, $\delta_x$ is a vector in the tangent space $\mathcal{T}_x \X$ at $x$. The dynamics of this virtual displacement are given by:
\[
    \dot{\delta}_x = \dfrac{\partial f}{\partial x} \delta_x,
\]
where $\partial f/\partial x$ is the Jacobian of $f$. The dynamics of the infinitesimal squared distance $\delta_x^T\delta_x$ between these two trajectories is then given by:
\[
    \dfrac{d}{dt}\left( \delta_x ^T \delta_x \right) = 2 \delta_x ^T \dfrac{\partial f}{\partial x} \delta_x.
\]
Then, if the (symmetric part) of the Jacobian matrix $\partial f/\partial x$ is \emph{uniformly} negative definite, i.e., 
\[
    \sup_{x} \bar{\lambda}\left(\dfrac{1}{2}\wwidehat{\dfrac{\partial f(x)}{\partial x}}\right) \leq -\lambda < 0,
\]
where $\wwidehat{(\cdot)} := (\cdot) + (\cdot)^T$, $\lambda > 0$, one has that the squared infinitesimal length $\delta_x^T\delta_x$ is exponentially convergent to zero at rate $2\lambda$. By path integration of $\delta_x$ between \emph{any} pair of trajectories, one has that the distance between any two trajectories shrinks exponentially to zero. The vector field $f$ is thereby referred to be \emph{contracting at rate $\lambda$}, and $\lambda$ is referred to as the \emph{contraction rate}.

Contraction metrics generalize this observation by considering as infinitesimal squared length distance, a symmetric positive definite function $V(x,\delta_x) = \delta_x^T M(x)\delta_x$, where $M: \X \rightarrow \Sjpp_n$ is a mapping from $\X$ to the set of uniformly positive definite $n\times n$ symmetric matrices. Formally, $M(x)$ may be interpreted as a Riemannian metric tensor, endowing the space $\X$ with the Riemannian squared length element $V(x,\delta_x)$. A fundamental result in contraction theory~\citep{LohmillerSlotine1998} is that \emph{any} contracting system admits a contraction metric $M(x)$ such that the associated function $V(x,\delta_x)$ satisfies:
\[
    \dot{V}(x,\delta_x) \leq - 2\lambda V(x,\delta_x), \quad \forall (x,\delta_x) \in \mathcal{T}\X,
\]
for some positive contraction rate $\lambda$. Thus, the function $V(x,\delta_x)$ may be interpreted as a \emph{differential Lyapunov function}. 

\subsection{Control Contraction Metrics}

Control contraction metrics (CCMs) generalize contraction analysis to the controlled dynamical setting, in the sense that the analysis searches \emph{jointly} for a controller design and the metric that describes the contraction properties of the resulting closed-loop system. Consider a control affine dynamical system of the form in~\eqref{cntrl_affine}. To define a CCM, analogously to the previous section, we first analyze the variational dynamics, i.e., the dynamics of an infinitesimal displacement $\delta_x$:
\begin{equation}
	\ddx= \overbrace{\bigg(\dfrac{\partial f(x)}{\partial x}  + \sum_{j=1}^m u^j \dfrac{\partial b_j(x)}{\partial x}\bigg)}^{:= A(x,u)}\delta_{x}+ B(x)\delta_{u},
\label{var_dyn_c}
\end{equation}
where $\delta_u$ is an infinitesimal (virtual) control vector at $u$ (i.e., $\delta_u$ is a vector in the control input tangent space, i.e., $\reals^m$). A uniformly positive definite matrix-valued function $M(x)$ is a CCM for the system $\{f,B\}$ if there exists a function $\delta_u(x,\dx,u)$ such that the function $V(x,\dx) = \dx^T M(x) \dx$ satisfies
\begin{equation}
\begin{split}
    \dot{V}(x,\dx,u) &= \delta_{x}^{T}\left(\partial_{f+Bu}M(x)+ \wwidehat{M(x)A(x,u)} \right) \delta_{x} + 2 \delta_{x}^{T}M(x)B(x)\delta_{u} \\
    &\leq -2\lambda V(x,\dx) \quad \forall (x,\dx) \in \mathcal{T}\X,\ u \in \U.
\end{split}
\label{V_dot}
\end{equation}
Given the existence of a CCM, one can then construct an exponentially stabilizing (in the sense of~\eqref{exp_stab}) feedback controller $k(x^*,x)$ as described in Appendix~\ref{ccm_appendix}.

Some important observations are in order. First, the function $V(x,\dx)$ may be interpreted as a differential CLF, in that there exists a stabilizing differential controller $\delta_u$ that stabilizes the variational dynamics~\eqref{var_dyn_c} in the sense of~\eqref{V_dot}. Second, and more importantly, we see that by stabilizing the variational dynamics (essentially an infinite family of linear dynamics in $(\delta_x,\delta_u)$) pointwise, everywhere in the state-space, we obtain a stabilizing controller for the original nonlinear system. Crucially, this is an exact stabilization result, not one based on local linearization-based control. Consequently, one can show several useful properties, such as invariance to state-space transformations~\citep{ManchesterSlotine2017} and robustness~\citep{SinghMajumdarEtAl2017,ManchesterSlotine2018}.  Third, the CCM approach only requires a weak form of controllability, and therefore is not restricted to feedback linearizable (i.e., invertible) systems. 

\section{CCM Constrained Dynamics Learning}\label{sec:prob}

Leveraging the characterization of stabilizability via CCMs, we can now formalize our dynamics learning problem as follows. Given a supervised dataset of demonstration tuples $\{(\xs,\us,\dot{x}_i)\}_{i=1}^{N}$, we wish to learn the dynamics functions $f(x)$ and $B(x)$ in~\eqref{cntrl_affine}, subject to the constraint that there exists a CCM $M(x)$ for the learned dynamics. \revision{That is, the CCM $M(x)$ plays the role of a \emph{certificate} of stabilizability for the learned dynamics.}

As shown in~\citep{ManchesterSlotine2017}, a necessary and sufficient characterization of a CCM $M(x)$ for the system $\{\hat{f},\hat{B}\}$ is given in terms of its dual $W(x):= M(x)^{-1}$ by the following two conditions:
\begin{align}
	 \hat{B}_{\perp}^{T}\left( \partial_{\hat{b}_j}W(x) - \wwidehat{\dfrac{\partial \hat{b}_j(x)}{\partial x}W(x)} \right)\hat{B}_{\perp}= 0, \ j = 1,\ldots, m \quad &\forall x \in \X,
\label{killing_A} \\
	   \underbrace{\hat{B}_{\perp}(x)^{T}\left(-\partial_{\hat{f}}W(x) + \wwidehat{\dfrac{\partial \hat{f}(x)}{\partial x}W(x)} + 2\lambda W(x) \right)\hat{B}_{\perp}(x)}_{:=\Fl(x;\hat{f},W)} \prec 0, \quad &\forall x \in \X, \label{nat_contraction_W}
\end{align}
where $\hat{B}_{\perp}$ is the annihilator matrix for $\hat{B}$, i.e., $\hat{B}(x)^T \hat{B}_\perp(x) = 0$ for all $x$. In the definition above, we write $\Fl(x;\hat{f},W)$ since $\{\hat{f},W\}$ will be optimization variables in our formulation while $\lambda$ is treated as a hyper-parameter. Thus, the learning task reduces to finding the functions $\{\hat{f},\hat{B},W\}$ that jointly satisfy the above constraints, while minimizing an appropriate regularized regression loss function. Formally, problem~\eqref{prob_gen} can be re-stated as:
\begin{subequations}\label{prob_gen2}
\begin{align}
&\min_{\substack{\hat{f} \in \mathcal{H}^{f}, \hat{b}_j \in \mathcal{H}^{B}, j =1,\ldots,m \\ W \in \mathcal{H}^W \\ \wl, \wu \in \reals_{>0}}} && \overbrace{\sum_{i=1}^{N} \left\| \hat{f}(\xs) + \hat{B}(\xs) \us - \dot{x}_i \right\|^2  + \mu_f \| \hat{f} \|^2_{\mathcal{H}^f} + \mu_b \sum_{j=1}^{m} \| \hat{b}_j \|^2_{\mathcal{H}^B}}^{:= J_d(\hat{f},\hat{B})} + \nonumber \\
& \qquad && + \underbrace{(\wu-\wl) +  \mu_w \|W\|^2_{\mathcal{H}^W}}_{:=J_m(W,\wl,\wu)}  \\
&\qquad \qquad \text{s.t.} && \text{eqs.~\eqref{killing_A},~\eqref{nat_contraction_W}} \quad \forall x \in \X, \\
& && \wl I_n \preceq W(x) \preceq \wu I_n \quad \forall x \in \X, \label{W_unif}
\end{align}
\end{subequations}
where $\mathcal{H}^f$ and $\mathcal{H}^B$ are appropriately chosen $\Y$-valued function classes on $\X$ for $\hat{f}$ and $\{\hat{b}_j\}_{j=1}^{m}$ respectively, and $\mathcal{H}^W$ is a suitable $\Sjpp_n$-valued function space on $\X$. The objective is composed of a dynamics term $J_d$ -- consisting of regression loss and regularization terms, and a metric term $J_m$ -- consisting of a condition number surrogate loss on the metric $W(x)$ and a regularization term. The metric cost term $\wu-\wl$ is motivated by the observation that the state tracking error (i.e., $\|x(t)-x^*(t)\|^2$) in the presence of bounded additive disturbances is proportional to the ratio $\wu/\wl$; see~\citep{SinghMajumdarEtAl2017}.

Notice that the coupling constraint~\eqref{nat_contraction_W} is a bi-linear matrix inequality in the decision variables  $\hat{f}$ and $W$. Thus, at a high-level, a solution algorithm must consist of alternating between two convex sub-problems, defined by the objective and decision variable pairs $(J_d, \{\hat{f},\hat{B},\lambda\})$ and $(J_m, \{W,\wl,\wu\})$. We refer to the first sub-problem as the \emph{dynamics} sub-problem, and the second as the \emph{metric} sub-problem. These sub-problems are described in full in Section~\ref{sec:soln}. 

\begin{remark} 
While one may include $\lambda$ as part of the overall optimization problem; indeed this is the case in the robust planning setting in~\citep{SinghMajumdarEtAl2017}, doing so would require either (i) adding $-\lambda$ to the dynamics sub-problem cost function, or (ii) line-searching over $\lambda$ within the metric sub-problem. The first option detracts from the primary objective of the dynamics sub-problem, which is to minimize regression error, while the second option induces needless complexity. Thus, in this work, $\lambda$ is held fixed as a tuning parameter.
\end{remark}

\section{CCM Regularized Dynamics Learning}\label{sec:reg}

When performing dynamics learning on a system that is a priori \emph{known} to be exponentially stabilizable with rate at least $\lambda$, the constrained problem formulation in~\eqref{prob_gen2} follows naturally given the assured \emph{existence} of a CCM. Albeit, the infinite-dimensional nature of the constraints is a considerable technical challenge, falling under the class of \emph{semi-infinite} optimization~\citep{HettichKortanek1993}. Alternatively, for systems that are not globally exponentially stabilizable in $\X$, one can imagine that such a constrained formulation may lead to adverse bias effects on the learned dynamics model. 

Thus, in this section we propose a relaxation of problem~\eqref{prob_gen2} motivated by the concept of regularization. Specifically, constraints~\eqref{killing_A} and~\eqref{nat_contraction_W} capture this notion of stability of infinitesimal deviations \emph{at all points} in the space $\X$. They stem from requiring that $\dot{V} \leq -2\lambda V(x,\dx)$ in~\eqref{V_dot} when $\dx^T M(x) B(x) = 0$, i.e., when $\delta_u$ can have no effect on $\dot{V}$. This is nothing but the standard control Lyapunov inequality, applied to the differential setting. Constraint~\eqref{killing_A} sets to zero, the terms in~\eqref{V_dot} affine in $u$, while constraint~\eqref{nat_contraction_W} enforces this ``natural" stability condition. 

The simplifications we make are (i) relax constraints~\eqref{nat_contraction_W} and~\eqref{W_unif} to hold pointwise over some \emph{finite} constraint set $X_c \in \X$, and (ii) assume a specific sparsity structure for the input matrix estimate $\hat{B}(x)$. We discuss the pointwise relaxation here; the sparsity assumption on $\hat{B}(x)$ is discussed in the following section.

First, from a purely mathematical standpoint, the pointwise relaxation of~\eqref{nat_contraction_W} and \eqref{W_unif} is motivated by the observation that, as the CCM-based controller is exponentially stabilizing, we only require the differential stability condition to hold locally (in a tube-like region) with respect to the provided demonstrations. By continuity of eigenvalues for continuously parameterized entries of a matrix, it is sufficient to enforce the matrix inequalities at a sampled set of points~\citep{Lax2007}.

Second, enforcing the existence of such an ``approximate" CCM seems to have an impressive regularization effect on the learned dynamics that is more meaningful than standard regularization techniques used in for instance, ridge or lasso regression. Specifically, problem~\eqref{prob_gen2}, and more generally, problem~\eqref{prob_gen} can be viewed as the \emph{projection} of the best-fit dynamics onto the set of stabilizable systems. This results in dynamics models that jointly balance regression performance and stabilizablity, ultimately yielding systems whose generated trajectories are notably easier to track. This effect of regularization is apparent in the simulation results presented in Section~\ref{sec:pvtol_sim}, and the hardware testbed results in Section~\ref{sec:pvtol_exp}. 

\revision{Practically, the finite constraint set $X_c$, with cardinality $N_c$, includes all $\xs$ in the supervised regression training set of $\{(\xs,\us,\dot{x}_i)\}_{i=1}^{N}$ tuples. However, as the LMI constraints are \emph{independent} of $\us$ and $\dot{x}_i$, the set $X_c$ is chosen as a strict superset of $\{\xs\}_{i=1}^{N}$ (i.e., $N_c > N$) by randomly sampling additional points within $\X$, drawing parallels with \emph{semi-supervised learning}.}

\subsection{Sparsity of Input Matrix} 
\label{sec:B_simp}

While a pointwise relaxation for the matrix inequalities is reasonable, one cannot apply such a relaxation to the exact equality condition in~\eqref{killing_A}. Thus, the second simplification made is the following assumption, reminiscent of control normal form equations.
\begin{assumption}\label{ass:B_simp}
Assume $\hat{B}(x)$ to take the following sparse representation:
\begin{equation}
    \hat{B}(x) = \begin{bmatrix} O_{(n-m)\times m} \\ \bs(x) \end{bmatrix},
\label{B_simp}
\end{equation}
where $\bs(x)$ is an invertible $m\times m$ matrix for all $x\in \X$. 
\end{assumption}
For the assumed structure of $\hat{B}(x)$, a valid $\hat{B}_{\perp}$ matrix is then given by:
\begin{equation}
    \hat{B}_{\perp} = \begin{bmatrix} I_{n - m} \\ O_{m \times (n-m)} \end{bmatrix},
    \label{B_perp}
\end{equation}
where $I$ and $O$ are the identity and zero matrices respectively. Therefore, constraint~\eqref{killing_A} simply becomes:
\[
	\partial_{\hat{b}_j} W_{\perp} (x) = 0, \quad j = 1,\ldots,m,
\]
where $W_{\perp}$ is the upper-left $(n-m)\times (n-m)$ block of $W(x)$. Assembling these constraints for the $(p,q)$ entry of $W_{\perp}$, i.e., $w_{\perp_{pq}}$, we obtain:
\[
	 \begin{bmatrix} \dfrac{ \partial w_{\perp_{pq}} (x) }{\partial x^{(n-m)+1}} & \cdots & \dfrac{\partial w_{\perp_{pq}} (x) }{\partial x^{n}} \end{bmatrix} \bs(x) = 0.
\]
Since the matrix $\bs(x)$ in~\eqref{B_simp} is assumed to be invertible, the \emph{only} solution to this equation is $\partial w_{\perp_{pq}}/ \partial x^i = 0$ for $i \in \{(n-m)+1,\ldots,n\}$, and all $(p,q) \in \{1,\ldots,(n-m)\}$. That is, $W_{\perp}$ cannot be a function of the last $m$ components of $x$ -- an elegant simplification of constraint~\eqref{killing_A}. 

Let us justify the structural assumption in~\eqref{B_simp}. Physically, this assumption is not as mysterious as it appears. Indeed, consider the standard dynamics form for mechanical systems:
\[
	H(q) \ddot{q} + C(q,\dot{q}) \dot{q} + g(q) = B(q) u,
\]
where $q \in \reals^{n_q}$ is the configuration vector, $H \in \Sjpp_{n_q}$ is the inertia matrix, $C(q,\dot{q})\dot{q}$ contains the centrifugal and Coriolis terms, $g$ are the gravitational terms, $B \in \reals^{n_q \times m}$ is the (full rank) input matrix mapping, and $u \in \reals^{m}$ is the input; for fully actuated systems, $\mathrm{rank}(B) = m = n_q$. For underactuated systems, $m < n_q$. By rearranging the configuration vector~\citep{Spong1998,Olfati-Saber2001,ReyhanogluSchaftEtAl1999}, one can partition $q$ as $(q_u, q_a)$ where $q_u \in \reals^{n_q - m}$ represents the unactuated degrees of freedom and $q_a \in \reals^{m}$ represents the actuated degrees of freedom. Applying this partitioning to the dynamics equation above yields:
\[
	\begin{bmatrix} H_{uu} (q) & H_{ua}(q) \\ H_{ua}(q) & H_{aa} (q)\end{bmatrix} \begin{bmatrix} \ddot{q}_u \\ \ddot{q}_a \end{bmatrix} + \begin{bmatrix} \tau_u(q,\dot{q}) \\ \tau_a (q,\dot{q}) \end{bmatrix} = \begin{bmatrix} O_{(n_q-m)\times m} \\ \bs(q) \end{bmatrix} u,
\]
where $\bs \in \reals^{m \times m}$ is an invertible square matrix. As observed in~\citep{ReyhanogluSchaftEtAl1999}, a substantial class of underactuated systems can be represented in this manner. Of course, fully actuated systems (i.e., $m=n_q$) also take this form. Thus, by taking as state $x = (q, p) \in \reals^{n}$ where $p = H(q) \dot{q}$ is momentum (so that $n = 2n_q$), the dynamics can be written as:
\begin{equation}
	\dot{x} = \begin{bmatrix} \dot{q} \\ \dot{p} \end{bmatrix} = \begin{bmatrix} H^{-1} (q) p \\ \dot{H}(q) \dot{q} - \tau(q,\dot{q}) \end{bmatrix}  + \begin{bmatrix} O_{(n-m) \times m} \\ \bs(q) \end{bmatrix}u.
\label{under_a}
\end{equation}
Notice that the input matrix takes the desired normal form in~\eqref{B_simp}. To address the apparent difficult of working with the state representation of $(q,p)$ (when usually only measurements of $(q,\dot{q},\ddot{q})$ are typically available \emph{and} the inertia matrix $H(q)$ is unknown), one may leverage the following result from~\citep{ManchesterSlotine2017}:

\begin{theorem}[CCM Invariance to Diffeomorphisms]\label{thm:ccm_inv}
Suppose there exists a valid CCM $M_x(x)$ with respect to the state $x$. Then, if $z = \psi(x)$ is a diffeomorphism, then the CCM conditions also hold with respect to state $z$ with metric $M_z(z) = \Psi(z)^{-T}M_x(x)\Psi(z)^{-1}$, where $\Psi(z) = \partial \psi(x)/\partial x$ is evaluated at $x = \psi^{-1}(z)$.
\end{theorem}

Thus, for the (substantial) class of underactuated systems of the form~\eqref{under_a}, one would solve problem~\eqref{prob_gen2} by fitting the dynamics terms $\{H,\tau,b\}$ using the state representation $x = (q,\dot{q})$, and leveraging Theorem~\ref{thm:ccm_inv} and the previous estimate of $H$ to enforce the matrix inequality constraints using the state representation $(q,p)$. This allows us to borrow several existing results from adaptive control on estimating mechanical system by leveraging the known linearity of $H(q)$ in terms of unknown mass property parameters multiplying known physical basis functions. We leave this extension however, to future work, and from hereon assume the structural form given in~\eqref{B_perp} for the estimated input matrix. 

\subsection{Finite-dimensional Optimization}\label{sec:finite}

We now present a tractable finite-dimensional optimization for solving problem~\eqref{prob_gen2} under the two simplifying assumptions \revision{introduced in the previous sections}. For ease of presentation, we outline the final optimization formulation here and provide the derivation, which relies extensively on vector-valued Reproducing Kernel Hilbert Spaces (RKHS), in the following section. 

\medskip

\noindent {\bf Parametrization:} Model $\hat{f}, \{\hat{b}_j\}_{j=1}^{m}$ and $\{w_{ij}\}_{i,j=1}^{n}$ as a linear combination of features. That is, 
\begin{align}
    \hat{f}(x) &= \Phi_f(x)^T \alpha, \label{param_1}\\
    \hat{b}_j(x) &= \Phi_b(x)^T \beta_j  \quad j \in \{1,\ldots, m\}, \\
    w_{ij}(x) &= \begin{cases} \hat{\phi}_w(x)^T \hat{\theta}_{ij} &\text{ if }\quad  (i,j) \in \{1,\ldots,n-m\}, \\
    \phi_w(x)^T \theta_{ij} &\text{ else}, \label{param_2}
    \end{cases}
\end{align}
where $\alpha \in \reals^{d_f}$, $\beta_j \in \reals^{d_b}$, $\hat{\theta}_{ij}, \theta_{ij} \in \reals^{d_w}$ are constant vectors to be optimized over, and $\Phi_f : \X \rightarrow \reals^{d_f\times n}$, $\Phi_b : \X \rightarrow \reals^{d_b \times n}$, $\hat{\phi}_w : \X \rightarrow \reals^{d_w}$ and $\phi_w : \X \rightarrow \reals^{d_w}$ are a priori chosen feature mappings. To enforce the sparsity structure in~\eqref{B_simp}, the feature matrix $\Phi_b$ must have all 0s in its first $n-m$ columns. The features $\hat{\phi}_w$ are distinct from $\phi_w$ in that the former are only a function of the first $n-m$ components of $x$ (as per Section~\ref{sec:B_simp}).
While one can use any function approximator (e.g., neural nets), we motivate this parameterization from a perspective of RKHS -- see Section~\ref{sec:deriv}.

\medskip 

\noindent {\bf Optimization:} Given positive regularization constants $\mu_f, \mu_b, \mu_w$ and positive tolerances $\epsilon_\lambda, \delta_{\wl}, \epsilon_{\wl}$, solve:
\begin{subequations}\label{learn_finite}
\begin{alignat}{2}
    &\min_{\alpha,\beta_j, \hat{\theta}_{ij}, \theta_{ij}, \wl, \wu} \quad   &&\overbrace{\sum_{i=1}^{N} \| \hat{f}(\xs)+\hat{B}(\xs)u_i - \dot{x}_i \|^2 + \mu_f \|\alpha\|^2 + \mu_b \sum_{j=1}^{m} \|\beta_j\|^2}^{:=\hat{J}_d}  \nonumber \\
    & \quad && \quad + \underbrace{(\wu-\wl) +  \mu_w\sum_{i,j} \|\tilde{\theta}_{ij}\|^2}_{:=\hat{J}_m}  \\
    &\qquad \qquad \text{s.t.} \quad && \mathcal{F}_{\lambda+\epsilon_{\lambda}}(\xs;\alpha,\tilde{\theta}_{ij}) \preceq 0 \quad \forall \xs \in X_c, \label{nat_finite} \\
    &\quad && (\wl + \epsilon_{\wl})I_{n} \preceq W(\xs) \preceq \wu I_n \quad \forall \xs \in X_c, \label{uniform_finite} \\
    &\quad && \theta_{ij} = \theta_{ji}, \  \hat{\theta}_{ij} = \hat{\theta}_{ji} \label{sym_finite} \\
    &\quad && \wl \geq \delta_{\wl}, \label{tol_finite}
\end{alignat}
\end{subequations}
where $\tilde{\theta}_{ij}$ is used as a placeholder for $\theta_{ij}$ and $\hat{\theta}_{ij}$ to simplify notation, and we use $\hat{J}_d$ and $\hat{J}_m$ to distinguish them from the functional equivalents $J_d$ and $J_m$ in problem~\eqref{prob_gen2}. We wish to highlight the following key points regarding problem~\eqref{learn_finite}. Constraints~\eqref{nat_finite} and~\eqref{uniform_finite} are the pointwise relaxations of~\eqref{nat_contraction_W} and~\eqref{W_unif} respectively;~\eqref{sym_finite} enforces the symmetry of $W(x)$, and~\eqref{tol_finite} imposes some tolerance requirements to ensure a well conditioned solution. Additionally, the tolerances $\epsilon_{\delta}$ and $\epsilon_{\wl}$ are used to account for the pointwise relaxations of the matrix inequalities. A key challenge is to efficiently solve this constrained optimization problem, given a potentially large number of constraint points in $X_c$. In Section~\ref{sec:soln}, we present an iterative algorithm and an adaptive constraint sampling technique to solve problem~\eqref{learn_finite}.

\section{Derivation of Problem~\eqref{learn_finite}}\label{sec:deriv}

To go from the general problem definition in~\eqref{prob_gen2} to the finite dimensional problem in~\eqref{learn_finite}, we first must define appropriate function classes for $\hat{f}$, $\hat{b}_j$, and $W$. We will do this using the framework of RKHS. We first provide a brief introduction to RKHS theory. 

\subsection{Reproducing Kernel Hilbert Spaces}

{\bf Scalar-valued RKHS}: Kernel methods~\citep{ScholkoepfSmola2001} constitute a broad family of non-parametric modeling techniques for solving a range of problems in machine learning. A scalar-valued positive definite kernel function $k : \X \times \X \mapsto \reals$ generates a RKHS $\mathcal{H}_k$ of functions, with the nice property that if two functions are close in the distance derived from the norm (associated with the Hilbert space), then their pointwise evaluations are close at all points. This continuity of evaluation functionals has the far-reaching consequence that norm-regularized learning problems over RKHSs admit finite dimensional solutions via Representer Theorems. The kernel $k$ may be (non-uniquely) associated with a higher-dimensional embedding of the input space via a feature map, $\phi: \reals^n \mapsto \reals^D$, such that $k(x, z) = \phi(x) ^T \phi(z)$, where $D$ is infinite for universal kernels associated with RKHSs that are dense in the space of square-integrable functions. Standard regularized linear statistical models in the embedding space implicitly provide nonlinear inference with respect to the original input representation.  In a nutshell, kernel methods provide a rigorous algorithmic framework for deriving nonlinear counterparts of a whole array of linear statistical techniques, e.g., for classification, regression, dimensionality reduction, and unsupervised learning. For further details, we point the reader to~\citep{HearstDumaisEtAl1998}. 

{\bf Vector-valued RKHS}: Dynamics estimation is a vector-valued learning problem. Such problems can be naturally formulated in terms of  vector-valued generalizations of RKHS concepts. The theory and formalism of vector-valued RKHS can be traced as far back as the work of Laurent Schwartz~\citep{Schwartz1964}, with applications ranging from solving partial differential equations to machine learning. Informally, we say that that ${\cal H}$ is an RKHS of $\reals^n$-valued maps if for any $v\in \reals^n$, the linear functional that maps $f \in {\cal H}$ to $v^T f(x)$ is continuous.

More formally, denote the standard inner product on $\Y$ as $\ip{\cdot}{\cdot}$, and let $\Y(\X)$ be the vector space of all functions $f : \X \rightarrow \Y$ and let $\Lin(\Y)$ be the space of all bounded linear operators on $\Y$, i.e., $n\times n$ matrices. A function $K : \X \times \X \rightarrow \Lin(\Y)$ is an \emph{operator-valued positive definite kernel} if for all $(x,z) \in \X \times \X$: $K(x,z)^T = K(z,x)$, and for all finite set of points $\{x_i\}_{i=1}^{N} \in \X$ and $\{y_i\}_{i=1}^{N} \in \Y$, 
\[
	\sum_{i,j=1}^{N} \ip{y_i}{K(x_i,x_j)y_j} \geq 0.
\]
Given such a $K$, we define the unique $\Y$-valued RKHS $\Hk \subset \Y(\X)$ with reproducing kernel $K$ as follows. For each $x \in \X$ and $y\in \Y$, define the function $K_{x} y = K(\cdot, x) y \in \Y(\X)$. That is, $K_x y (z) = K(z,x) y$ for all $z \in \X$. Then, the Hilbert space $\Hk$ is defined to be the completion of the linear span of functions $\{K_x y \ | \ x \in \X, y \in \Y\}$ with inner product $\langle\cdot,\cdot\rangle_{\Hk}$ between functions $f  = \sum_{i=1}^{N} K_{x_i} y_i,\ g = \sum_{j=1}^{N'} K_{z_j} w_j \in \Hk$, defined as:
\[
	\ip{f}{g}_{\Hk} = \sum_{i=1}^{N}\sum_{j=1}^{N'} \ip{y_i}{K(x_i,z_j) w_j},
\] 
and function norm $\|f\|^2_{\Hk}$ given as $\ip{f}{f}_{\Hk}$. Analogous to the canonical reproducing property of scalar-valued RKHS kernels, i.e., 
\begin{equation}
	h(x) = \ip{h}{k(x,\cdot)}_{\mathcal{H}_k} \quad \forall h \in \mathcal{H}_k,
\label{rk_rep}
\end{equation}
the matrix-valued kernel $K$ satisfies the following reproducing property:
\begin{equation}
	\ip{f(x)}{y} = \ip{f}{K_x y}_{\Hk} \quad \forall (x,y) \in \X \times \Y, \forall f \in \Hk,
\label{RK_rep}
\end{equation}
and is thereby referred to as the \emph{reproducing kernel} for $\Hk$. As our learning problem involves the Jacobian of $f$, we will also require a representation for the derivative of the kernel matrix. Accordingly, suppose the kernel matrix $K$ lies in $\mathcal{C}^{2}(\X \times \X)$. For any $j \in \{1,\ldots, n\}$, define the matrix functions $\frac{\partial}{\partial s^j} K: \X \times \X \rightarrow \Lin(\Y)$ and $\frac{\partial}{\partial r^j} K : \X \times \X \rightarrow \Lin(\Y)$ as
\[
	\dfrac{\partial}{\partial s^j} K(z,x) := \left.\dfrac{\partial}{\partial s^j} K(r , s) \right|_{r = z, s = x}   \quad \dfrac{\partial}{\partial r^j} K(z,x) := \left. \dfrac{\partial}{\partial r^j} K(r,s) \right|_{r = z,s=x},
\]
where the derivative of the kernel matrix is understood to be element-wise. Define the $\mathcal{C}^1$ function $\partial_j K_x y \in \Y(\X)$ as
\[
	\partial_j K_x y(z) := \dfrac{\partial}{\partial s^j} K(z, x)y \quad \forall z \in \X.
\]
The following result provides a useful reproducing property for the derivatives of functions in $\Hk$ that will be instrumental in deriving the solution algorithm.

\begin{theorem}[Derivative Properties on $\Hk$ \citep{MicheliGlaunes2014}] \label{thm:RK_der}
Let $\Hk$ be a RKHS in $\Y(\X)$ with reproducing kernel $K \in \mathcal{C}^2(\X \times \X)$. Then, for all $(x,y) \in \X \times \Y$:
\begin{enumerate}
	\item[(a)] $\partial_j K_x y \in \Hk$ for all $j = 1,\ldots,n$. 
	\item[(b)] The following derivative reproducing property holds for all $j = 1,\ldots, n$:
	\[
		\ip{\dfrac{\partial f(x)}{\partial x^j}}{y} = \ip{f}{ \partial_j K_x y}_{\Hk} \quad \forall f \in \Hk.
	\]
\end{enumerate}
\end{theorem}

As mentioned in Section~\ref{sec:prob}, the bi-linearity of constraint~\eqref{nat_contraction_W} forces us to adopt an alternating solution strategy whereby in the dynamics sub-problem, $W$ is held fixed and we minimize $J_d$ with respect to $\{\hat{f},\hat{B}\}$. In the metric sub-problem, $\{\hat{f},\hat{B}\}$ are held fixed and we minimize $J_m$ with respect to $\{W,\wl,\wu\}$.

In the following we derive several useful representer theorems to characterize the solution of the two sub-problems, under the two simplifying assumptions introduced in Section~\ref{sec:finite}. 

\subsection{Dynamics Sub-Problem}\label{sec:dyn_sub}

Let $K^f$ be the reproducing $\mathcal{C}^2$ kernel for an $\Y$-valued RKHS $\Hk^f$ and let $K^B$ be another reproducing kernel for an $\Y$-valued RKHS $\Hk^B$. Define the finite-dimensional subspaces:
\begin{align}
	\Vf &:= \left\{ \sum_{i=1}^{N_c} K^f_{\xs} a_i + \sum_{i=1}^{N_c} \sum_{p=1}^{n} \partial_p K^f_{\xs} a'_{ip},\quad a_i, a'_{ip} \in \reals^n \right\} \subset \Hk^f \ , \label{rep} \\
	\Vb &:= \left\{ \sum_{i=1}^{N_c} \Kb_{\xs} c_i + \sum_{i=1}^{N_c} \sum_{p=1}^{n} \partial_p \Kb_{\xs} c'_{ip},\quad c_i, c'_{ip} \in \reals^n \right\} \subset \Hk^B \ .\label{rep_b}
\end{align}
Note that all $\xs$ taken from the training dataset of $(\xs,\us,\dot{x}_i)$ tuples are a subset of $X_c$.

\begin{theorem}[Representer Theorem for $f,B$]\label{thm:rep_dyn}
Suppose the reproducing kernel $K^B$ is chosen such that all functions $g \in \Hkb$ satisfy the sparsity structure $g^j(x) = 0$ for $j = 1,\ldots,n-m$. Consider then the pointwise-relaxed dynamics sub-problem for problem~\eqref{prob_gen2}:
\begin{alignat}{2}
&\min_{\substack{\hat{f} \in \Hk^{f}, \ \hat{b}_j \in \Hk^{B}, j =1,\ldots,m }} \qquad  && J_d(\hat{f},\hat{B})  \nonumber \\
 & \qquad \qquad \mathrm{s.t.} && \Fl(\xs; \hat{f},W) \preceq 0 \quad \forall \xs \in X_c \ . \label{lmi_pw1}
\end{alignat}
Suppose the feasible set for the LMI constraint is non-empty. Denote $f^*$ and $b_j^*, j = 1,\ldots, m$ as the optimizers for this sub-problem. Then, $f^* \in \Vf$ and $\{b^*_j\}_{j=1}^{M} \in \Vb$.
\end{theorem}
\begin{proof}
For a fixed $W$, the constraint is convex in $\hat{f}$ by linearity of the matrix $\Fl$ in $\hat{f}$ and $\partial \hat{f}/\partial x$. By assumption~\eqref{ass:B_simp}, and the simplifying form for $B_{\perp}$, the matrix $\Fl$ is additionally independent of $B$. Then, by strict convexity of the objective functional, there exists unique minimizers $f^*, \{b_j\}^* \in \Hk$, provided the feasible region is non-empty~\citep{KurdilaZabarankin2006}. 

Since $\Vf$ and $\Vb$ are closed, by the Projection Theorem, $\Hk^f = \Vf \oplus \Vf^{\perp}$ and $\Hkb = \Vb \oplus \Vb^{\perp}$. Thus, any $g \in \Hk^f$ may be written as $g^{\Vf} + g^{\perp}$ where $g^{\Vf} \in \Vf$, $g^{\perp} \in \Vf^\perp$, and $\ip{g^{\Vf}}{g^{\perp}}_{\Hk^f} = 0$. Similarly, any $g \in \Hkb$ may be written as $g^{\Vb} + g^{\perp}$ where $g^{\Vb} \in \Vb$, $g^{\perp} \in \Vb^\perp$, and $\ip{g^{\Vb}}{g^{\perp}}_{\Hkb} = 0$.

Let $f = f^{\Vf} + f^{\perp}$ and $b_j = b_{j}^{\Vb} + b_{j}^\perp$, and define 
\[
	\begin{split}
	H^{\V}(x,u) &= f^{\Vf}(x) + \sum_{j=1}^{m} u^j b_j^{\Vb} (x) \\ 
	H^{\perp}(x,u) &= f^{\perp}(x) + \sum_{j=1}^{m} u^j b_j ^\perp (x).
	\end{split}
\]
We can re-write $J_d(f,B)$ as:
\[
\begin{split}
	J_d(f,B) &= \sum_{i=1}^{N} \ip{H^{\V} (x_i, u_i) - \dot{x}_i}{ H{^\V} (x_i,u_i) - \dot{x}_i} + 2\ip{H^{\V} (x_i, u_i) - \dot{x}_i}{H^\perp (x_i,u_i)}  \\
	&\qquad \qquad + \ip{H^\perp(x_i, u_i)}{H^\perp(x_i, u_i)} +  \mu_f \left( \|f^{\Vf}\|_{\Hk^f}^2 + \|f^\perp \|_{\Hk^f}^2 \right)  \\
	&\qquad \qquad + \mu_b \left( \sum_{j=1}^{m} \|b_j^{\Vb} \|_{\Hkb}^2 + \|b_j^\perp \|_{\Hkb}^2 \right). \\
\end{split}
\]
Leveraging the reproducing property in~\eqref{RK_rep}, 
\[
	\begin{split}
	&\ip{H^{\V} (\xs, u_i) - \dot{x}_i}{H^\perp (\xs,u_i)}  \\
				&\qquad = \ip{f^\perp (\xs) + \sum_{j=1}^{m} u_i^j b_j^\perp(\xs)}{H^\V (\xs, u_i) - \dot{x}_i} \\
				&\qquad = \underbrace{\ip{f^\perp}{K_{\xs}^f \left(H^\V (\xs, u_i) - \dot{x}_i \right)}_{\Hk^f}}_{=\,0} + \underbrace{\ip{\sum_{j=1}^{m} u_i^j b_j^\perp}{\Kb_{\xs} \left(H^\V (\xs, u_i) - \dot{x}_i \right)}_{\Hkb}}_{=\,0} 
	\end{split}
\]
since $K_{\xs}^f \left(H^\V (\xs, u_i) - \dot{x}_i \right) \in \Vf$, $\Vb^\perp$ is closed under addition, and $\Kb_{\xs} \left(H^\V (\xs, u_i) - \dot{x}_i \right) \in \Vb$. Thus, $J_d(f,B)$ simplifies to
\begin{equation}
\begin{split}
	\sum_{i=1}^{N} \left\| H^{\V} (\xs, u_i) - \dot{x}_i \right\|^2 &+ \mu_f  \|f^{\Vf}\|_{\Hk^f}^2 + \mu_b \sum_{j=1}^{n} \|b_j^{\Vb} \|_{\Hk^B}^2  + \\
	&+\left[ \sum_{i=1}^{N}  \left\| H^\perp (\xs, u_i) \right\|^2 +  \mu_f \|f^\perp \|_{\Hk^f}^2  + \mu_b\sum_{j=1}^{m}  \|b_j^\perp \|_{\Hk^B}^2 \right].
\end{split}
\label{obj_simp}
\end{equation}
Now, the $(p,q)$ element of $\partial_f W(\xs)$ takes the form 
\[
	\ip{ \frac{ \partial w_{pq} (\xs)}{\partial x}}{f (\xs)} = \ip{f}{K_{\xs}^f\frac{ \partial w_{pq} (\xs)}{\partial x}}_{\Hk^f} =  \ip{f^{\Vf}}{K^f_{\xs}\frac{ \partial w_{pq} (\xs)}{\partial x}}_{\Hk^f}.
\]
Column $p$ of $\frac{\partial f(\xs)}{\partial x} W(\xs)$ takes the form
\[
	\begin{split}
	\sum_{j=1}^{n} w_{j p} (\xs) \dfrac{\partial f(\xs)}{\partial x^j}  = \begin{bmatrix} \sum_{j=1}^{n}  w_{jp}(\xs) \ip{ \frac{\partial f(\xs)}{\partial x^j} }{ e_1} \\ \vdots \\ 	\sum_{j=1}^{n}  w_{jp}(\xs) \ip{ \frac{\partial f(\xs)}{\partial x^j} }{ e_n}	\end{bmatrix} & = 
												       	 \begin{bmatrix} \sum_{j=1}^{n}  w_{jp}(\xs) \ip{ f }{\partial_j K^f_{\xs} e_1}_{\Hk^f} \\ \vdots \\ 	\sum_{j=1}^{n}  w_{jp}(\xs) \ip{ f }{\partial_j K^f_{\xs} e_n}_{\Hk^f}	\end{bmatrix} \\
					&=  \begin{bmatrix} \sum_{j=1}^{n}  w_{jp}(\xs) \ip{ f^{\Vf} }{\partial_j K^f_{\xs} e_1}_{\Hk^f} \\ \vdots \\ 	\sum_{j=1}^{n}  w_{jp}(\xs) \ip{ f^{\Vf} }{\partial_j K^f_{\xs} e_n}_{\Hk^f}	\end{bmatrix},
	\end{split}												       	 
\]
where $e_i$ is the $i^{\text{th}}$ standard basis vector in $\reals^n$. Thus, $f^\perp$ plays no role in pointwise relaxation of constraint~\eqref{nat_contraction_W} and thus does not affect problem feasibility.
Given the assumed structure of functions in $\Hk^B$ as provided in the theorem statement, the matrix in~\eqref{B_perp} is a valid annihilator matrix for $\hat{B}(x)$. Consequently, $b^{\perp}$ also has no effect on problem feasibility. Thus, by non-negativity of the term in the square brackets in~\eqref{obj_simp}, we have that the optimal $f$ lies in $\Vf$ and all optimal $\{b_j\}_{j=1}^{m}$ lie within $\Vb$. 
\end{proof}

The key consequence of this theorem is the reduction of the infinite-dimensional search problem for the functions $f$ and $b_j, j=1,\ldots,m$ to a finite-dimensional \emph{convex} optimization problem for the constant vectors $a_i,a'_{ip},\{c_i^{(j)}, c_{ip}^{(j)'}\}_{j=1}^{m} \in \reals^n$, by choosing the function classes $\mathcal{H}^f = \Hk^{f}$ and $\mathcal{H}^B = \Hk^{B}$. Next, we characterize the optimal solution to the metric sub-problem.

\subsection{Metric Sub-Problem}\label{sec:met_sub}

By the simplifying assumption in Section~\ref{sec:B_simp}, constraint~\eqref{killing_A} requires that $W_{\perp}$ be only a function of the first $(n-m)$ components of $x$. Thus, define $\kw : \X \times \X \rightarrow \reals$ as a \emph{scalar} reproducing kernel with associated real-valued scalar RKHS $\Hw$. Additionally, define $\kwp : \X \times \X \rightarrow \reals$ as another scalar reproducing kernel with associated real-valued scalar RKHS $\Hwp$. In particular, $\kwp$ is only a function of the first $(n-m)$ components of $x \in \X$ in both arguments. Let $\kw_x$ and $\kwp_x$ denote $\kw(x,\cdot)$ and $\kwp(x,\cdot)$ respectively. 

Define the kernel derivative functions:
\[
    \partial_j \kw_{x}(z) := \left.\dfrac{\partial \kw}{\partial r^j} \kw(r,s)\right|_{(r=x,s=z)} \quad \partial_j \kwp_{x}(z) := \left.\dfrac{\partial \kwp}{\partial r^j} \kw(r,s)\right|_{(r=x,s=z)}  \quad \forall z \in \X.
\]
From~\citep{Zhou2008}, it follows that the kernel derivative functions satisfy the following two properties, similar to Theorem~\ref{thm:RK_der}:
\[
    \begin{split}
            \partial_j \kw_{x} \in \Hw \quad \forall j = 1,\ldots,n,\  x \in \X \\
            \dfrac{\partial h}{\partial x^j}(x) = \ip{h}{\partial_j \kw_x}_{\Hw} \quad \forall h \in \Hw.
    \end{split}
\]
A similar property holds for $\partial_j \kwp_x$ and $\Hwp$. Consider then the following finite-dimensional spaces:
\begin{align}
	\V_{\kw} &:= \left\{ \sum_{i=1}^{N_c}  a_i \kw_{\xs} + \sum_{i=1}^{N_c} \sum_{p=1}^{n}  a_{ip}' \,  \partial_p \kw_{\xs} ,\quad a_i, a_{ip}' \in \reals \right\} \subset \Hw \\
	\V_{\kwp} & := \left\{ \sum_{i=1}^{N_c} c_i \kwp_{\xs} + \sum_{i=1}^{N_c} \sum_{p=1}^{n-m}  c_{ip}' \,  \partial_p \kwp_{\xs} ,\quad c_i, c_{ip}' \in \reals \right\} \subset \Hwp 
\end{align}
and the proposed representation for $W(x)$:
\begin{equation}
	\begin{split}
	W(x) =  &\sum_{i=1}^{N_c} \hat{\Theta}_i \kwp_{\xs} (x) + \sum_{i=1}^{N_c}\sum_{j=1}^{n-m} \hat{\Theta}'_{ij} \partial_{j}\kwp_{\xs} (x)  \\
		+  & \sum_{i=1}^{N_c} \Theta_i \kw_{\xs} (x) + \sum_{i=1}^{N_c}\sum_{j=1}^{n} \Theta'_{ij} \partial_{j}\kw_{\xs} (x),
	\end{split}
\label{W_rep}
\end{equation}
where $\hat{\Theta}, \hat{\Theta}' \in \Sj_{n}$ are constant symmetric matrices with non-zero entries only in the top-left $(n-m)\times(n-m)$ block, and $\Theta,\Theta' \in \Sj_{n}$ are constant symmetric matrices with zero entries in the top-left $(n-m)\times(n-m)$ block.

\begin{theorem}[Representer Theorem for $W$]\label{thm:rep_metric}
Consider the pointwise-relaxed metric sub-problem for problem~\eqref{prob_gen2}:
\begin{alignat}{2}
&\min_{\substack{w_{pq} \in \Hwp, (p,q) \in \{1,\ldots,(n-m)\} \\ w_{pq} \in \Hw \text{ else} \\ \wl,\wu \in \reals_{>0}}} \qquad && J_m(W,\wl,\wu)  \nonumber \\
&\qquad \qquad \quad \mathrm{s.t.} \qquad && \Fl(\xs; \hat{f}, W) \preceq 0, \quad \forall \xs \in X_c, \label{lmi_pw2} \\
&\qquad && \wl I_n \preceq W(\xs) \preceq \wu I_n, \quad \forall \xs \in X_c, \label{lmi_pw3}
\end{alignat}
Suppose the feasible set of the above LMI constraints is non-empty. Denote $W^*$ as the optimizer for this sub-problem. Then, $W^*$ takes the form given in~\eqref{W_rep}.
\end{theorem}
\begin{proof}
Notice that while the regularizer term is strictly convex, the surrogate loss function for the condition number is affine. However, provided the feasible set is non-empty, there still exists a minimizer (possibly non-unique) for the above sub-problem. 

Since $\V_{\kw}$ and $\V_{\kwp}$ are closed, by the Projection Theorem, $\Hw = \V_{\kw} \oplus \V_{\kw}^{\perp}$ and $\Hwp = \V_{\kwp} \oplus \V_{\kwp}^{\perp}$. Thus, any $h\in \Hw$ may be written as $h^{\V_{\kw}} + h^{\perp}$ where $h^{\V_{\kw}} \in \V_{\kw}$, $h^{\perp} \in \V_{\kw}^{\perp}$ and $\ip{h^{\V_{\kw}}}{h^{\perp}}_{\Hw} = 0$. A similar decomposition property holds for $\Hwp$. 

Now, consider the following chain of equalities for the $(p,q)^{\text{th}}$ element of $\partial_{f} W_{\perp} (\xs)$:
\[
	\begin{split} 
		\partial_{f} w_{\perp_{pq}}(\xs) &=\sum_{j=1}^{n-m} \dfrac{\partial w_{\perp_{pq}} (\xs)}{\partial x^j} f^j (\xs) \\
								&= \sum_{j=1}^{n-m} \ip{w_{\perp_{pq}}}{\partial_{j}\kwp_{\xs}}_{\Hwp} f^j(\xs) \\
								&= \ip{w_{\perp_{pq}}}{ \sum_{j=1}^{n-m} f^j(\xs) \partial_{j}\kwp_{\xs}}_{\Hwp} \\
								&= \ip{w^{\V_{\kwp}}_{\perp_{pq}}}{ \sum_{j=1}^{n-m} f^j(\xs) \partial_{j}\kwp_{\xs}}_{\Hwp} ,
	\end{split}
\]
since $\sum_{j=1}^{n-m} f^j(\xs) \partial_{j}\kwp_{\xs} \in \V_{\kwp}$. Trivially, the $(p,q)$ element of $W(\xs)$ takes the form
\[
		w_{pq}(\xs) = \begin{cases}  \ip{w^{\V_{\kwp}}_{pq}}{\kwp_{\xs}}_{\Hwp} &\text{ if } (p,q) \in \{1,\ldots,(n-m)\} \\
						 	    \ip{w^{\V_\kw}_{pq}}{\kw_{\xs}}_{\Hw} &\text{ else.}
					\end{cases}
\]
Thus, constraints~\eqref{nat_contraction_W} and uniform definiteness at the constraint points in $X_c$ can be written in terms of functions in $\V_\kw$ and $\V_{\kwp}$ alone. By strict convexity of the regularizer, and recognizing that $W(x)$ is symmetric, the result follows. 
\end{proof}

Similar to the previous section, the key property here is the reduction of the infinite-dimensional search over $W(x)$ to the finite-dimensional convex optimization problem over the constant symmetric matrices $\Theta_i,\Theta_{ij}',\hat{\Theta}_i,\hat{\Theta}'_{ij}$ by choosing the function class for the entries of $W$ using the scalar-valued RKHS.

At this point, both sub-problems are finite-dimensional convex optimization problems. Crucially, the only simplifications made are those given in Section~\ref{sec:finite}. However, a final computational challenge here is that the number of parameters scales with the number of training ($N$) and constraint ($N_c$) points. This is a fundamental challenge in all non-parametric methods. In the next section we present the dimensionality reduction techniques used to alleviate these issues.

\subsection{Approximation via Random Matrix Features}
\label{sec:feat}

The size of the problem using full matrix-valued kernel expansions grows rapidly in $N_c \cdot n$, the number of constraint points times the state dimensionality. This makes training slow for even moderately long demonstrations in low-dimensional settings. The induced dynamical system is slow to evaluate and integrate at inference time. Random feature approximations to kernel functions have been extensively used to scale up training complexity and inference speed of kernel methods~\citep{RahimiRecht2007,AvronSindhwaniEtAl2016} in a number of applications~\citep{HuangAvronEtAl2014}.  The quality of approximation can be explicitly controlled by the number of random features. In particular, it has been shown~\citep{RahimiRecht2008} that any function in the RKHS associated with the exact kernel can be approximated to arbitrary accuracy by a linear combination of a sufficiently large number of random features. 

These approximations have only recently been extended to matrix-valued kernels~\citep{Minh2016,BraultHeinonenEtAl2016}. Given a matrix-valued kernel $K$, one defines an appropriate matrix-valued feature map $\Phi : \X \rightarrow \reals^{d\times n}$ with the property 
\[
    K(x,z) \approx \Phi(x)^T \Phi(z),
\]
where $d$ controls the quality of this approximation. A canonical example is the Gaussian separable kernel $K_{\sigma}(x,z) := e^{\frac{-\|x-z\|_2^2}{\sigma^2}} I_n$ with feature map:
\[
    \Phi(x)  = \dfrac{1}{\sqrt{s}}\begin{bmatrix} \cos(\omega_1^T x) \\ \sin(\omega_1^T x) \\ \vdots \\ \cos(\omega_s^T x) \\ \sin(\omega_s^T x)    \end{bmatrix} \otimes I_n,
\]
where $\omega_1,\ldots,\omega_s$ are i.i.d. draws from $\mathcal{N}(0,\sigma^{-2}I_n)$ and $\otimes$ denotes the Kronecker product. 

By construction, the linear span $\{K_z y \mid z \in \X, y \in \Y\}$ is dense in $\Hk$. Thus, any function $g$ in the associated RKHS $\Hk$ can be arbitrarily well-approximated by the expansion $\sum_{j}^{N'} K_{z_j} y_j$, which, in turn, may be approximated as
\[
    g(x) \approx \sum_{j=1}^{N'} K(x,z_j) y_j \approx \sum_{j=1}^{N'} \Phi(x)^T \Phi(z_j) y_j = \Phi(x)^T \alpha,
\]
where $\alpha = \sum_{j=1}^{N'} \Phi(z_j) y_j \in \reals^d$. 
Applying this approximation for the spaces\footnote{To apply this decomposition to $W(x)$, we simply leverage vector-valued feature maps for each entry of $W(x)$.} $\Hk^f$, $\Hk^B$, $\Hw$ and $\Hwp$, we obtain finite-dimensional representations of the optimal $f,\{b_j\}, W$ that scale with the order of the matrix feature map, instead of $N_c \cdot n$. Indeed, these are the representations defined in eqs.~\eqref{param_1}--\eqref{param_2}. To complete the derivation of the cost terms $\hat{J}_d$ and $\hat{J}_m$ in problem~\eqref{learn_finite}, we address the functional regularization terms as follows. From the definition of the inner product in $\Hk$:
\[ 
    \|g\|_{\Hk}^2 = \sum_{i,j=1}^{N'} \ip{y_i}{K(z_i,z_j) y_j} \approx \sum_{i,j=1}^{N'} \ip{\Phi(z_i)y_i}{\Phi(z_j) y_j} = \ip{\sum_{i=1}^{N'}\Phi(z_i)y_i}{\sum_{j=1}^{N'}\Phi(z_j)y_j} = \|\alpha\|^2 \ .
\]
Prior to proceeding to the solution algorithm, we summarize the content since Section~\ref{sec:prob} until now. First, starting from the infinite-dimensional \emph{non-convex} constrained optimization over function spaces, i.e., problem~\eqref{prob_gen2}, we split the problem into alternating between two infinite-dimensional \emph{convex} sub-problems. Second, we derived the representation of the optimal solution of each sub-problem, i.e., Theorems~\ref{thm:rep_dyn} and~\ref{thm:rep_metric}, by leveraging the mild structural assumption on the input matrix, relaxation of the infinite-dimensional constraints to sampling-based constraints, and restriction of the function spaces to RKHS. Third, we used the universal approximation property of random matrix feature maps to reduce the dimensionality of the representation of the optimal solutions. Fourth, and finally, problem~\eqref{learn_finite} gives the final finite-dimensional re-formulation of problem~\eqref{prob_gen2} using this feature representation and the sample-based constraints.


\section{Solution Algorithm} \label{sec:soln}

The fundamental structure of the solution algorithm consists of alternating between the dynamics and metric sub-problems derived from problem~\eqref{learn_finite}. However, there are three fundamental challenges in solving these problems. First, enforcing the constraints exactly requires a feasible initialization (either with the dynamics or metric function parameters). Given our use of random matrix feature approximations, this is simply intractable. Second, even with an initial feasible guess, due to the alternation, one may lose feasibility within either of the sub-problems stated in Theorems~\ref{thm:rep_dyn} and~\ref{thm:rep_metric}. This is a fundamental numerical challenge of bilinear optimization. Third, since the constraint set $X_c$ includes at least the state samples from the demonstration tuples (plus additional random samples from $\X$), enforcing the LMIs over all constraint points at each iteration is again computationally intractable. 

To address these challenges, we present a solution algorithm that (i) eliminates the need for a feasible initial guess, thereby permitting cold-starts, (ii) leverages sub-sampling and \emph{dynamic} updating of the constraint set for each iteration sub-problem to maintain tractability, and (iii) performs feasibility corrections using efficient, unconstrained second-order optimization. We first give the full formulation of the relevant optimization sub-problems and then provide a pseudocode summary of the algorithm at the end. 

Let $X_c^{(k)}$ denote the finite sample constraint set at iteration $k$. In particular, $X_c^{(k)} \subset X_c$ with $N_c^{(k)}:= |X_c^{(k)}|$ being ideally much less than $N_c$, the cardinality of the full constraint set $X_c$. Formally, each major iteration $k$ is characterized by four minor steps (sub-problems):

\begin{enumerate}
\item {\bf Finite-dimensional dynamics sub-problem}:
\begin{subequations} \label{finite_dyn}
\begin{alignat}{2}
    &\min_{\substack{\alpha,\beta_j, j=1,\ldots,m, \\ s \geq 0}} \qquad && \sum_{i=1}^{N} \| \hat{f}(\xs)+\hat{B}(\xs)u_i - \dot{x}_i \|^2 + \mu_f \|\alpha - \alpha^{(k-1)}\|^2 \nonumber \\
    &\quad && + \mu_b \sum_{j=1}^{m} \|\beta - \beta_j^{(k-1)}\|^2+ \mu_s\|s\|_1 \\
    &\qquad \text{s.t.}  && \mathcal{F}_{\lambda+\epsilon_{\lambda}}(\xs;\alpha,\tilde{\theta}^{(k-1)}_{ij}) \preceq s(\xs)I_{n-m} \quad \forall \xs \in X_c^{(k)} \\
    & \qquad && s(\xs) \leq \bar{s}^{(k-1)}  \quad \forall \xs \in X_c^{(k)},
\end{alignat}
\end{subequations}
where $\{\alpha^{(k-1)}, \beta_j^{(k-1)}, \tilde{\theta}_{ij}^{(k-1)} \}$ are the dynamics and metric parameters obtained from the previous major iteration\footnote{We set $\alpha^{(0)}, \{\beta_j^{(0)}\}_{j=1}^{m}, \tilde{\theta}_{ij}^{(0)} = 0$.}, and $\mu_s \in (0,1)$ is an additional regularization parameter for $s$, a non-negative slack vector in $\reals^{N_c^{(k)}}$. The quantity $\bar{s}^{(k-1)}$ is defined as
\[
    \begin{split}
    \bar{s}^{(k-1)} &:= \max_{\xs \in X_c} \bar{\lambda} \left(\mathcal{F}_{\lambda+\epsilon_{\lambda}}^{(k-1)}(\xs)\right), \quad \text{where} \\
    \mathcal{F}_{\lambda+\epsilon_{\lambda}}^{(k-1)}(\xs) &:= \mathcal{F}_{\lambda+\epsilon_{\lambda}}(\xs;\alpha^{(k-1)},\tilde{\theta}^{(k-1)}_{ij}).
    \end{split}
\]
That is, $\bar{s}^{(k-1)}$ captures the worst violation for the stability LMI over the entire constraint set $X_c$, given the parameters at the end of iteration $k-1$. Denote $\alpha^{(k)}, \{\beta_j^{(k)}\}_{j=1}^{m}$ to be the optimizers of this sub-problem.

\item {\bf Compute an upper bound $\bar{s}^{(k)'}$ on the maximum allowed violation of the stability LMI for the current constraint set $X_c^{(k)}$}:
\begin{leftbox}
\begin{subequations} \label{metric_infeas}
\begin{alignat}{2}
&{\bf \texttt{while}} \quad  &&\min_{\xs \in X_c^{(k)}} \quad  \underline{\lambda} \left( W(\xs) \right) < \delta_{\wl}+\epsilon_{\wl} \\
& \qquad {\bf \texttt{do}} \quad  \min_{\tilde{\theta}_{ij}} \quad && \sum_{\xs \in X_c^{(k)}} \Bigg[ \psi_w \left( \bar{\lambda} \left( (\delta_{\wl}+\epsilon_{\wl})I_{n} - W(\xs) \right) \right) + \nonumber \\
& \qquad \qquad &&  \qquad \qquad +\mu_{w}' \psi_{\mathcal{F}} \left( \bar{\lambda} \left( \mathcal{F}_{\lambda+\epsilon_{\lambda}} (\xs; \alpha^{(k)},\tilde{\theta}_{ij} ) \right) \right) \Bigg] + \nonumber \\
& && + \mu_{w}' \sum_{i,j} \|\tilde{\theta}_{ij} - \tilde{\theta}_{ij}^{(k-1)} \|^2, \\
& \mu_{w}' \leftarrow 0.5 \mu_{w}' && 
\end{alignat}
\end{subequations}
\end{leftbox}
where $\mu_{w}' >0$ is a regularization parameter that is exponentially decayed until the \texttt{while} loop termination condition is met. The functions $\psi_w(\cdot)$ and $\psi_{\mathcal{F}}(\cdot)$ are penalty functions. Let $\tilde{\theta}_{ij}'$ be the metric parameters when the termination condition is met. We then set
\begin{equation}
	\bar{s}^{(k)'} := \max_{\xs \in X_c^{(k)}} \bar{\lambda} \left( \mathcal{F}_{\lambda+\epsilon_{\lambda}} (\xs ; \alpha^{(k)}, \tilde{\theta}'_{ij}) \right). 
\end{equation}
The primary objective of this step is to compute an upper bound on the stability LMI over the constraint set $X_c^{(k)}$ with respect to a dual metric that satisfies the positive definiteness condition at all points in this set. This upper-bound will then be used as a constraint within the metric sub-problem, described next. The reason one cannot use $\bar{s}^{(k-1)}$ as in the dynamics sub-problem is because this value is computed over the entire constraint set at the end of the previous iteration, with a metric that is potentially \emph{not} positive definite at all points. Since the metric sub-problem strictly enforces the positive definiteness constraint on $W$, $\bar{s}^{(k-1)}$ is not a valid upper bound on the stability LMI constraint, and can lead to infeasibility for the metric sub-problem. As long as there exists a set of $\tilde{\theta}_{ij}$ such that the dual metric satisfies the positive definiteness constraint at all points, the \texttt{while} loop above will always terminate given a sufficiently small $\mu_w'$. In our implementation, we initialize $\mu_w'$ to be equal to $\mu_w$. In Section~\ref{sec:metric_infeas} we provide an efficient second-order method for solving this \emph{unconstrained} optimization. 

\item {\bf Finite-dimensional metric sub-problem}:
\begin{subequations}\label{finite_met}
\begin{align}
    \min_{\tilde{\theta}_{ij},\wl,\wu,  s \geq 0} \quad & (\wu - \wl) + \mu_w \sum_{i,j} \|\tilde{\theta}_{ij} - \tilde{\theta}_{ij}^{(k-1)} \|^2 + (1/\mu_s)\|s\|_1 \\
    \text{s.t.} \quad & \mathcal{F}_{\lambda+\epsilon_{\lambda}} (\xs;\alpha^{(k)},\tilde{\theta}_{ij}) \preceq s(\xs)I_{n-m}  \quad \forall \xs \in X_c^{(k)} \\
    \quad & s(\xs) \leq \bar{s}^{(k)'} \quad \forall \xs \in X_c^{(k)} \\
    \quad &  (\wl + \epsilon_{\wl})I_{n} \preceq W(\xs) \preceq \wu I_n \quad \forall \xs \in X_c^{(k)}, \\
    \quad & \wl \geq \delta_{\wl}.
\end{align}
\end{subequations}

\item {\bf Update $X_c^{(k)}$ sub-problem}. 

Choose a tolerance parameter $\delta>0$. Then, define
    \begin{equation}
        \nu^{(k)}(\xs) := \max \left\{ \bar{\lambda} \left(\mathcal{F}_{\lambda+\epsilon_{\lambda}}^k(\xs)\right) , \bar{\lambda} \left((\delta_{\wl}+\epsilon_{\delta})I_n - W(\xs) \right) \right \} \quad \forall \xs \in X_c,
    \label{constraint_viol}
    \end{equation}
    and set
    \begin{equation}
        X_{c}^{(k+1)} :=  \left\{ \xs \in X_c^{(k)} : \nu^{(k)}(\xs) > -\delta \right\} \bigcup  \left\{\xs \in X_c \setminus X_c^{(k)} : \nu^{(k)}(\xs) > 0 \right\}. 
        \label{Xc_up}
    \end{equation}
Thus, in the update $X_c^{(k)}$ step, we balance addressing points where constraints are being violated ($\nu^{(k)} > 0$) and discarding points where constraints are satisfied with sufficient strict inequality ($\nu^{(k)}\leq -\delta$). This prevents overfitting to any specific subset of the constraint points. A potential variation to the union above is to only add up to $L$ constraint violating points from $X_c\setminus X_c^{(k)}$ (e.g., corresponding to the $L$ worst violators), where $L$ is a fixed positive integer. Indeed this is the variation used in our experiments and was found to be extremely efficient in balancing the size of the set $X_c^{(k)}$ and thus, the complexity of each iteration. This adaptive sampling technique is inspired by \emph{exchange algorithms} for semi-infinite optimization, as the one proposed in~\citep{ZhangWuEtAl2010} where one is trying to enforce the constraints at \emph{all} points in a compact set $\X$.
\end{enumerate}

The use of the modified regularization terms above (i.e., penalizing change in the parameters from the previous iteration as opposed to absolute penalty) is, in the spirit of trust region methods, to prevent large updates to the parameters due to the dynamically updating constraint set $X_c^{(k)}$. At the first major iteration of the problem however, with $\alpha^{(0)}, \{\beta_j^{(0)}\}_{j=1}^{m}, \tilde{\theta}_{ij} = 0$, this is the same as an absolute regularization term. To ensure non-degeneracy in the stability LMI constraint within the first dynamics sub-problem iteration, we initialize the algorithm with $W(\xs) = I_n$. The full pseudocode for the CCM-Regularized (CCM-R) dynamics learning algorithm is summarized in Algorithm~\ref{alg:final}.

\begin{algorithm}[h!]
  \caption{CCM - Regularized (CCM-R) Dynamics Learning}
  \label{alg:final}
  \begin{algorithmic}[1]
  \State {\bf Input:} Dataset $\{\xs,\us,\dot{x}_i\}_{i=1}^{N}$, constraint set $X_c$, regularization constants $\{\mu_f,\mu_b,\mu_w, \mu_s \}$, constraint tolerances $\{\epsilon_\lambda,\delta_{\wl},\epsilon_{\wl} \}$, discard tolerance parameter $\delta$, Initial \# of constraint points: $N_c^{(0)}$, Max \# iterations: $N_{\max}$, termination tolerance $\varepsilon$. 
   \State $k \leftarrow 1$, \texttt{converged} $\leftarrow$ \textbf{false}, $W(x) \leftarrow I_n$, $ \{\alpha^{(0)}, \{\beta_j^{(0)}\}_{j=1}^{m}, \tilde{\theta}_{ij} \} \leftarrow 0$.
   \State $X_c^{(0)} \leftarrow \textproc{RandSample}(X_c,N_c^{(0)})$ \label{line:rand_samp_init}
   \While {$\neg \texttt{converged} \wedge k<N_{\max} $} 
    \State $\{\alpha^{(k)}, \{\beta_j^{(k)}\} \} \leftarrow \textproc{Solve}$~\eqref{finite_dyn}
    \State $\bar{s}^{(k)'} \leftarrow \textproc{Solve}$~\eqref{metric_infeas}
    \State $\{\tilde{\theta}_{ij}^{(k)},\wl,\wu\} \leftarrow \textproc{Solve}$~\eqref{finite_met}
    \State $X_c^{(k+1)}, \bar{s}^{(k)}, \nu^{(k)} \leftarrow$ \textproc{Update} $X_c^{(k)}$ using~\eqref{Xc_up}
    \State {\small $\Delta \leftarrow \max\left\{\|\alpha^{(k)}-\alpha^{(k-1)}\|_{\infty},\|\beta_j^{(k)}-\beta_j^{(k-1)}\|_{\infty},\|\tilde{\theta}_{ij}^{(k)}-\tilde{\theta}_{ij}^{(k-1)}\|_{\infty}\right\}$}
    \If{$\Delta < \varepsilon$ \textbf{or} $\nu^{(k)}(\xs) < \varepsilon \quad \forall \xs \in X_c$}
        \State \texttt{converged} $\leftarrow$ \textbf{true}.
    \EndIf
    \State $k \leftarrow k + 1$.
  \EndWhile
      \end{algorithmic}
\end{algorithm} 

\revision{Some comments are in order. First, convergence in Algorithm~\ref{alg:final} is declared if either progress in the solution variables stalls or all constraints are satisfied within tolerance. Using dynamic sub-sampling for the constraint set at each iteration implies that (i) the matrix function $W(x)$ at iteration $k$ resulting from variables $\tilde{\theta}_{ij}^{(k)}$ does \emph{not} have to correspond to a valid dual metric for the interim learned dynamics at iteration $k$, and (ii) a termination condition based on constraint satisfaction at all $N_c$ points is justified by the fact that at each iteration, we are solving relaxed sub-problems that collectively generate a sequence of lower-bounds on the overall objective. 

Second, as a consequence of this iterative procedure, the dual metric and contraction rate pair $\{W(x),\lambda\}$ do not possess any sort of ``control-theoretic'' optimality. For instance, in~\citep{SinghMajumdarEtAl2017}, for a known stabilizable dynamics model, both these quantities are optimized for robust control performance. In this work, these quantities are used solely as \emph{regularizers} to \emph{promote} stabilizability of the learned model. Following the recovery of a fixed dynamics model from the algorithm, one may leverage global CCM optimization algorithms as the one presented in~\citep{SinghMajumdarEtAl2017} to compute an optimal \emph{robust} CCM.}

\subsection{Solving the Upper Bound Sub-Problem}\label{sec:metric_infeas}

Step 2 in the iterative algorithm outlined in the previous section entails solving a non-smooth, but convex, unconstrained optimization to compute a viable upper bound on the stability LMI for its inclusion as a constraint within the metric sub-problem. While the problem can be formulated as an SDP through the use of epigraph LMIs, this would result in yet another computationally intensive step. Instead, recognizing that the purpose of this part of the algorithm is \emph{not} to update any problem variables, one can use faster unconstrained optimization techniques to approximately solve this problem. 

Specifically, we employ Newton descent to solve this problem with backtracking line search. The issue however, lies in taking the second order derivatives of the maximum eigenvalue function of an affinely parameterized matrix (recall $W$ is linear in $\tilde{\theta}_{ij}$ and and $\Fl$ is linear in $\tilde{\theta}_{ij}$ for fixed $\alpha$). For a given affinely parameterized symmetric matrix $G(\theta) \in \reals^{d\times d}$, the gradient and Hessian of $\bar{\lambda}(G(\theta))$ with respect to $\theta$ has the components~\citep{OvertonWomersley1995}:
\[
\begin{split}
	\dfrac{\partial \bar{\lambda}(G)}{\partial \theta^i}(\theta) &= \bar{v}_1^T \dfrac{\partial G(\theta)}{\partial \theta^i} \bar{v}_1, \\
	\dfrac{ \partial^2 \bar{\lambda}(G)}{\partial \theta^i \partial \theta^j}(\theta) &= 2 \sum_{k>1}^{d} \dfrac{\bar{v}_1^T \frac{\partial G(\theta)}{\partial \theta^i} \bar{v}_k \cdot \bar{v}_1^T \frac{\partial G(\theta)}{\partial \theta^j} \bar{v}_k }{\lambda_1 - \lambda_k},
\end{split}
\]
where we assume the eigenvalue ordering $\lambda_1 \geq \cdots \geq \lambda_d$, and $\bar{v}_1,\ldots,\bar{v}_d$ are the associated normalized eigenvectors. Clearly, the Hessian is undefined when the multiplicity of the max eigenvalue is greater than one. While there exist several works within the literature on modified parameterizations of the matrix to allow the use of second order methods, these techniques require knowing the multiplicity of the max eigenvalue at the optimal parameters -- which is impossible to guess in this case. 

Instead, we leverage a stochastic smoothing approximation borrowed from~\citep{DAspremontKaroui2014}, whereby one approximates the max eigenvalue of $G(\theta)$ as
\[
	\max_{i = 1,\ldots,k} \bar{\lambda} \left( G(\theta) + \dfrac{\sigma}{d} z_i z_i^T \right),
\] 
where $\sigma > 0$ is a small noise parameter, and $z_i$ are i.i.d. samples from $\mathcal{N}(0, I_d)$. In particular, $\bar{\lambda}\left( G(\theta) + \frac{\varepsilon}{n} z_i z_i^T \right)$ is unique with probability one, for any $z_i \sim \mathcal{N}(0,I_d)$ (see~\citep{DAspremontKaroui2014}, Prop.~3.3 and Lemma 3.4). While the algorithm in~\citep{DAspremontKaroui2014} leverages a finite-sample expectation of the randomized function above and its gradient, for our implementation, we simply leverage the uniqueness property induced by the use of the random rank one perturbation, and compute the gradient and Hessian expressions using the eigenvalue decomposition of the matrix $G(\theta) + (\varepsilon/d)zz^T$, where $z$ is a single Gaussian vector sample. Additionally, we employ early termination of the Newton iterations if the \texttt{while} loop termination condition is met by the current iterate. 


\subsection{Revisiting PVTOL} \label{sec:results}

We now revisit the results presented in Section~\ref{sec:pvtol_sim} and shed additional light on the performance of the CCM-R model. 

The training dataset was generated in three steps. First, a fixed set of waypoint paths in $(p_x,p_z)$ was randomly generated. Second, for each waypoint path, multiple smooth polynomial splines were fitted using a minimum-snap algorithm. To create variation amongst the splines, the waypoints were perturbed within Gaussian balls and the time durations for the polynomial segments were also randomly perturbed. Third, the PVTOL system was simulated with perturbed initial conditions and the polynomial trajectories as references, and tracked using a sub-optimally tuned PD controller; thereby emulating a noisy/imperfect demonstrator. These final simulated paths were sub-sampled at $0.1$ s resolution to create the datasets. The variations created at each step of this process were sufficient to generate a rich exploration of the state-space for training. 

The matrix feature maps used for all models (N-R, R-R, and CCM-R) are derived from the random matrix feature approximation of the Gaussian separable kernel; the relevant parameters are provided in Appendix~\ref{app:prob_params}. We enforced the CCM regularizing constraints for the CCM-R model at $N_c = 2426$ points in the state-space, composed of the $N$ demonstration points in the training dataset and randomly sampled points from $\X$ (recall that the CCM constraints do not require samples of $u,\dot{x}$). The contraction rate $\lambda$ was set to 0.1 and the termination tolerance $\varepsilon$ was set to $0.01$. All other tolerance parameters are provided in Appendix~\ref{app:prob_params}.

As the CCM constraints were relaxed to hold pointwise on the finite constraint set $X_c$ as opposed to everywhere on $\X$, in the spirit of viewing these constraints as regularizers for the model, we simulated both the R-R and CCM-R models using the time-varying Linear-Quadratic-Regulator (TV-LQR) feedback controller. This also helped ensure a more direct comparison of the quality of the learned models themselves, independently of the tracking feedback controller. The results are virtually identical using a tracking MPC controller and yield no additional insight.

Figure~\ref{fig:sim_train_curves} plots the training curves generated by the CCM-R algorithm for varying demonstration dataset sizes $N$. In particular, we plot (i) the evolution of the distribution of the constraint violation vector $\nu$ defined in~\eqref{constraint_viol} over the training constraint set $X_c$, and (ii) evolution of the regression error and fraction of violations (i.e., average number of points with $\nu > 0$) over an independent validation set, as a function of global iteration $k$. The dynamic constraint set was initialized with 250 points (sampled randomly from the 2426 constraint points in the full training constraint set), and over all iterations, the size of $X_c^{(k)}$ remained below 400 points -- a drastic factor of improvement over an intractable brute-force approach. For $N\in \{100,250,500\}$, the algorithm terminated with all constraints satisfied below the termination threshold of $\varepsilon = 0.01$. For $N=1000$, we manually terminated the algorithm after 12 global iterations after observing a stall in the number of violations. 

\begin{figure}[h]
\centering
\begin{subfigure}[t]{0.7\textwidth}
	\includegraphics[width=1\textwidth]{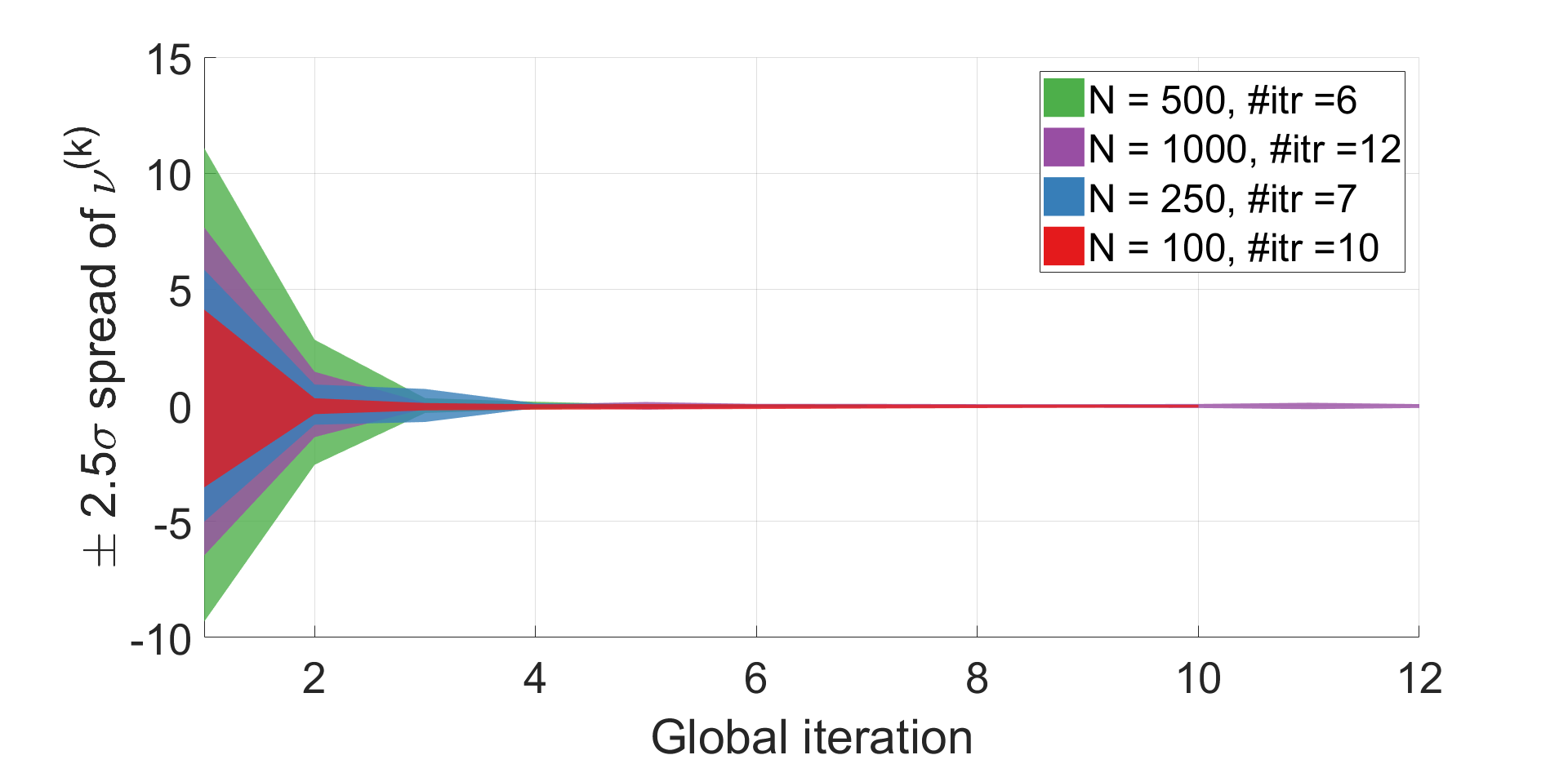}
	\caption{$\pm 2.5 \sigma$ spread of the constraint violation vector $\nu$ over \emph{training set} as a function of global iteration. $\#$\texttt{itr} denotes the total number of global iterations. }
	\label{fig:sim_curve_viol}
\end{subfigure} \qquad 
\begin{subfigure}[t]{0.7\textwidth}
	\includegraphics[width=1\textwidth]{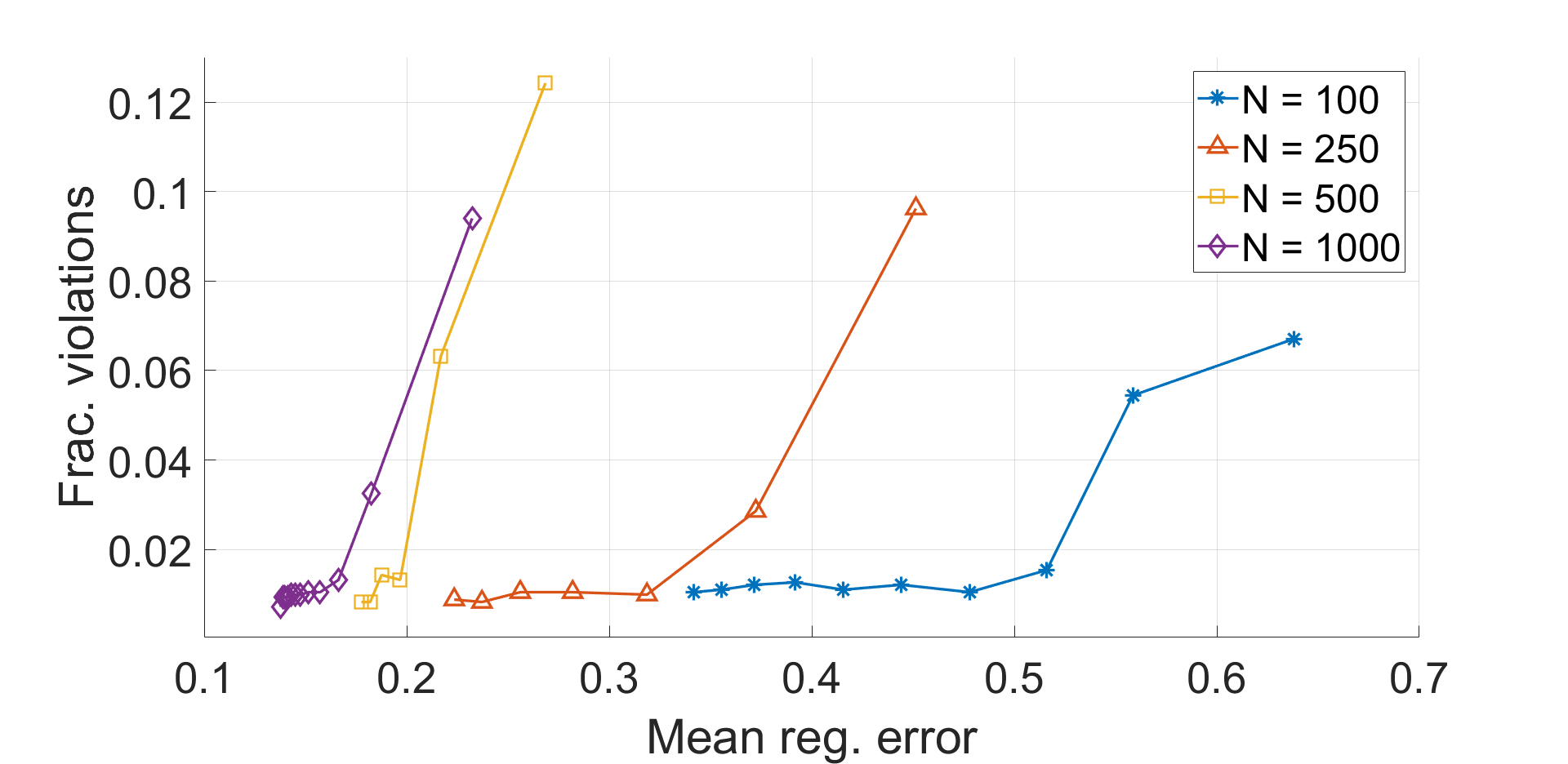}
	\caption{Evolution of mean regression error norm and fraction of violations over \emph{validation} set with global iteration number. The markers delineate the iterations. The curves proceed from right to left.}	
	\label{fig:sim_curve_reg}
\end{subfigure}
	\caption{Simulation data training curves for all CCM-R models.}
	\label{fig:sim_train_curves} 
\end{figure}
From Figure~\ref{fig:sim_curve_viol}, we see the effectiveness of the sub-sampling constraint set method in efficiently eliminating all training violations. Several comments are in order regarding Figure~\ref{fig:sim_curve_reg}. First, as the number of demonstration tuples $N$ goes down, the curves shift further to the right, indicating higher validation error, as expected. Second, the iterates seem to clearly exhibit a two-phase convergent behavior. During the first phase, the fraction of violations on the validation set drops rapidly and approaches a steady-state. During the second phase, the fraction of violations remains roughly constant as the validation error drops monotonically, with decreasing drop rate. The training and final validation errors for all models are summarized in Table~\ref{tab:model_reg}.

\begin{table}[H]\centering
\begin{tabular}{@{}lccccccccc@{}}
\toprule
& \multicolumn{2}{c}{\textbf{N-R}} & \phantom{a} & \multicolumn{2}{c}{\textbf{R-R}} & \phantom{a} &  \multicolumn{3}{c}{\textbf{CCM-R}} \\
\cmidrule{2-3} \cmidrule{5-6} \cmidrule{8-10}
$N$ & Train err.  & Val err. && Train err. & Val err. && Train err. & Val err. & Frac. viol. \\
\midrule
100 & 0.001 & 0.072 && 0.125 & 0.380 && 0.012 & 0.342 & 0.0008 (0.010) \\
250 & 0.003 & 0.088 && 0.125 & 0.243 && 0.010 & 0.223 & 0.0004 (0.009) \\
500 & 0.004 & 0.047 && 0.092 & 0.172 && 0.006 & 0.178 & 0.0016 (0.008) \\
1000 & 0.01 & 0.056 && 0.08 & 0.146 && 0.003 & 0.138 & 0.0021 (0.007) \\
\bottomrule
\end{tabular}
\caption{Comparison of average (over $N$ tuples) training and validation (over 2000 tuples) regression error norms for all 3 models. Also shown are the fraction of violations ($\nu>0$) on the training (validation) sets for the CCM-R models. The termination threshold for CCM-R training was $\nu < \varepsilon = 0.01$ for all points in the training constraint set. }
\label{tab:model_reg}
\end{table}

An interesting trend to observe here is that the training errors for the N-R and CCM-R models are significantly closer (by an order of magnitude) together than R-R and N-R. Consistent with the trend in Figure~\ref{fig:mu_vs_reg}, N-R features the smallest validation errors. Thus, while the CCM-R model manages to perform well in terms of regression error on the training set (almost on par with N-R), the resulting model does not suffer from the overfitting trend observed with the N-R model.

Surprisingly, the final validation errors of the CCM-R models and the R-R models are almost identical. {\it Despite this however, the CCM-R model significantly outperforms the R-R model, especially when learned from smaller supervised datasets}. What sets the CCM-R model apart is the manner in which the dynamics sub-problem improves the regression error over multiple iterations -- through the joint minimization of regression error and stability LMI violations. Effectively, as seen in Figure~\ref{fig:sim_curve_reg}, the iterates essentially move along a level set of the constraint violations, decreasing regression error while maintaining control over the \emph{stabilizability of the learned model}. This validates the benefit of using a control-theoretic regularization technique that is \emph{tailored} to the motion planning task, and the fragility of \emph{only} using traditional measures of performance (such as regression error) for evaluating learned dynamical models\footnote{Code for data generation, training, and evaluation is provided at \url{https://github.com/StanfordASL/SNDL}.}. 

While the analytical PVTOL model studied in simulation is indeed an example of a system that is CCM stabilizable (see~\citep{SinghMajumdarEtAl2017}), in the next section we deploy the algorithm on a full 3D quadrotor testbed, shown in Figure~\ref{fig:px4_stock}, with partially closed control loops to approximately emulate the PVTOL dynamics. 

\section{Validation on Quadrotor Testbed}\label{sec:pvtol_exp}

Our quadrotor testbed consists of (i) a standard DJI F330 frame, (ii) a Pixhawk autopilot running the estimator and lower-level thrust and angular rate controllers, and (iii) a companion on-board ODROID-XU4 computer running ROS nodes for motion planning and trajectory tracking controller\footnote{The code for these ROS nodes is available for download at \url{https://github.com/StanfordASL/asl_flight}.} (which generates the thrust and angular rate setpoints for Pixhawk). There is also an Optitrack motion capture system providing inertial position and yaw estimates at 120~Hz, which is fused with the onboard EKF on the Pixhawk.
\begin{figure}[h]
\centering
	\includegraphics[width=0.8\textwidth]{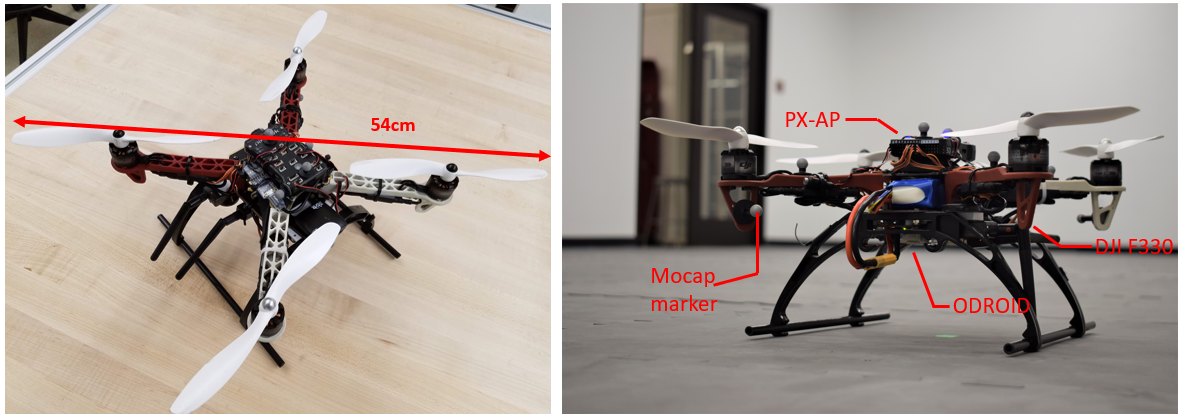}
	\caption{Quadrotor experimental platform, equipped with Pixhawk autopilot (PX-AP) for low-level (thrust and angular rate) control, and ODROID companion computer for planning and trajectory tracking control.}
	\label{fig:px4_stock}
\end{figure}

The objective of the flight experiments was to verify whether the performance trends observed in simulation for a relatively simple dynamics model, that was a priori known to be CCM stabilizable, carried over to a far more complex setting where this is unknown. The dynamics to be learned are for the motion of the quadrotor \emph{within the vertical plane}. Accordingly, the training data consisted of samples from trajectories flown within the plane autonomously using a nonlinear trajectory tracking controller. For evaluation, the generated planar trajectories will be tracked using a combination of an \emph{in-plane} controller, derived from the learned dynamics model, and an \emph{out-of-plane} controller derived from known principles (detailed in Section~\ref{sec:quad_oop}). The purpose of the out-of-plane controller is simply to keep the quad aligned and fixed within the vertical plane, thereby permitting an evaluation of the quality of the learned planar dynamics. 

This is a complex system subject to noise and complex nonlinear disturbances stemming from aerodynamic effects (e.g., ground-effect) and residual out-of-plane motion, and therefore constitutes a challenging evaluation benchmark.

\medskip

\noindent {\bf State and control parameterization}: We adopt the NED body-frame convention and the $XYZ$-Euler angle sequence parameterized by, in order, roll ($\phi$), pitch ($\theta$), and yaw ($\psi$) angles. For simplicity, we fixed the nominal yaw angle to $0$ and inertial $X$-position to be constant so that the planar motion is aligned with the inertial $Y$-$Z$ plane. The state representation used for the planar dynamics is therefore: $x = (p_y, p_z, \phi, v_y, v_z, \omega_x)$ where $(p_y,p_z)$ are the inertial position in the $Y$-$Z$ plane, $(v_y,v_z)$ are the body-frame velocities, and $\omega_x$ is the body-frame angular rate. The control inputs are set to be the desired net normalized thrust $\tau$ and desired body rate $\dot{\phi}_c$. These inputs are fed into the PX4 autopilot which uses an on-board faster rate control loop to realize these commands. 

The supervisory signal $\dot{x}$ is provided by $(\dot{p}_y, \dot{p}_z, \omega_x, \dot{v}_y, \dot{v}_z, \dot{\omega}_x)$. Note that by kinematics, $\omega_x = \dot{\phi}\cos(\theta)\cos(\psi) + \dot{\theta}\sin(\psi)$. Thus, any non-zero $\psi$, and/or non-zero $\dot{\theta}$ with non-zero $\psi$ can introduce bias errors into the training data. The signals $\dot{\omega}_x$ and $(\dot{v}_y, \dot{v}_z)$ were obtained by convolving a finite-difference derivative estimate through a moving-average filter. 

\subsection{Model Training}

\noindent{\bf Training data collection}: To collect the training data, we created a set of training trajectories within the plane consisting of clockwise and counter-clockwise circles at varying speeds and ``swoop" trajectories (see Figure~\ref{fig:circle_swoop_traces}) where the quadrotor swoops towards the ground in a parabolic arc and travels in a straight line close to the ground. The goal of the swoops was to learn data pertaining to ground-effect aerodynamic disturbances. All trajectories were tracked using a CCM-based feedback controller derived from a nominal model of a quadrotor that \emph{neglects any aerodynamic disturbances.} That is, the tracking controller and designed trajectories do not pre-compensate for any aerodynamic disturbances. Instead, we would like for such strategies to evolve directly from the learned model and subsequently derived tracking controller. 
\begin{figure}[h]
\centering
    \begin{minipage}[t][][b]{0.35\textwidth}
	    \centering\includegraphics[height=7cm,trim=0 0 5cm 0, clip]{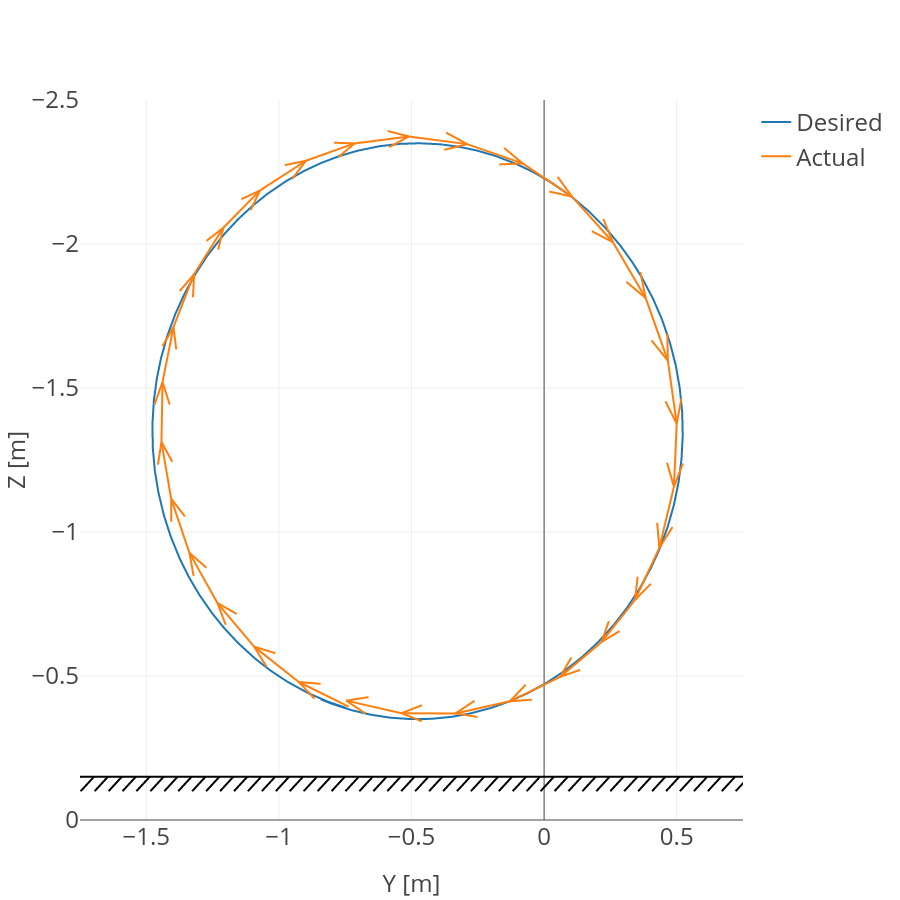}
	\end{minipage}\hspace{5em}
	\begin{minipage}[t][][b]{0.45\textwidth}
    	\centering\includegraphics[height=7cm]{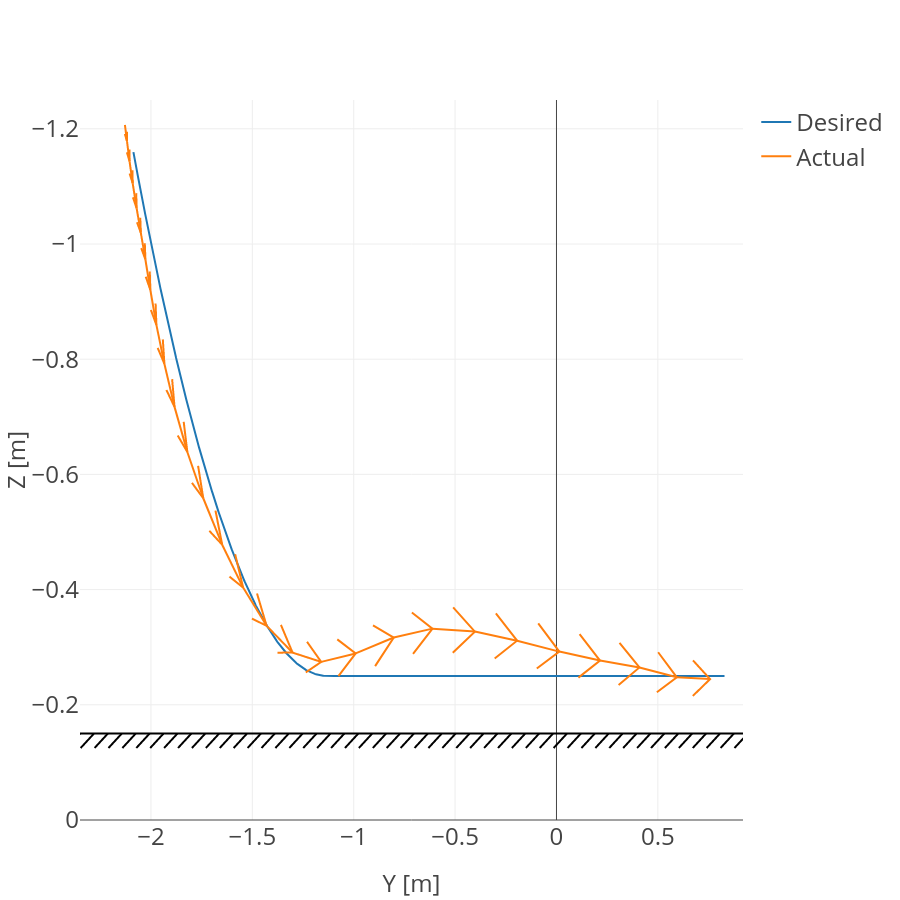}
	\end{minipage}
	\caption{Examples of flown trajectories with a 3D nonlinear tracking feedback controller to create the training dataset. Left: clockwise circle with radius 1 m, period 6 s, and a nominal speed of 1.05 m/s. Right: Swoop trajectory to excite the ground-effect aerodynamic disturbances. Notice the deviation of the quadrotor from the desired trajectory as it approaches the ground. Ideally, a model learned from such data should generate nominal control signals that compensate for such effects.}
	\label{fig:circle_swoop_traces}
\end{figure}

We collected data from 3 clockwise and 3 counter-clockwise 1~m radius circles at periods of $10, 8$ and $6$ seconds (corresponding to nominal speeds of 0.63, 0.79, and 1.05~m/s), and two symmetric swoop trajectories.  Figure~\ref{fig:circle_errors} shows the yaw and inertial $X$ tracking errors for the fastest clockwise training circle. Both these quantities pertain to the out-of-plane motion and need to be small in order to minimize any induced bias within the planar motion, which is significantly coupled with the out-of-plane motion. This is indeed observed to be the case, validating the suitability of the collected data to train a planar dynamics model. 
\begin{figure}[h]
    \centering
    \includegraphics[width=0.7\textwidth, trim=0 0 0 3cm, clip]{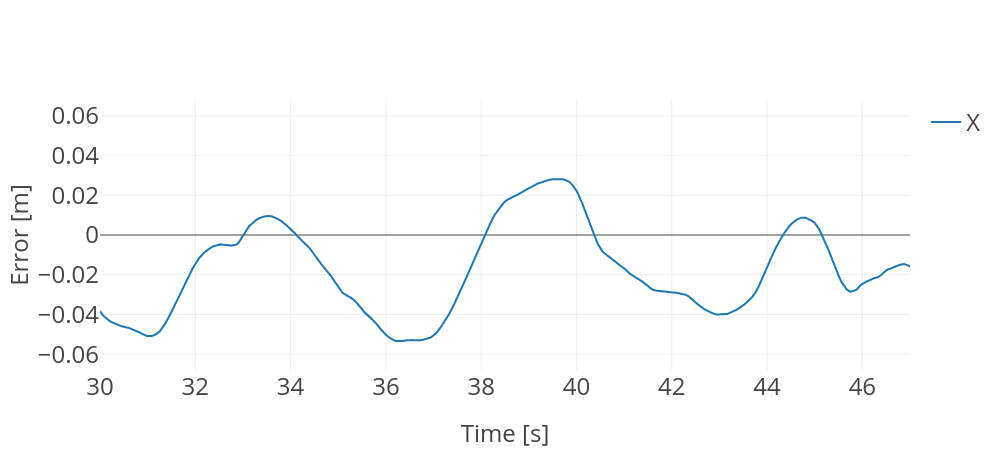}\\
    \includegraphics[width=0.7\textwidth, trim=0 0 0 3.25cm, clip]{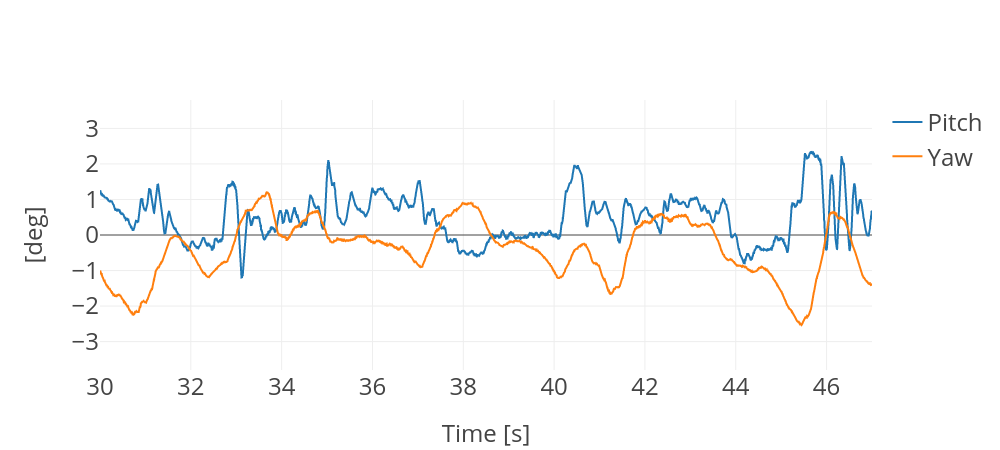}
    \caption{Out-of-plane errors, i.e., inertial $X$ motion and pitch and yaw angles. All errors are sufficiently small, validating the use of the remaining data to train a planar dynamics model.}
	\label{fig:circle_errors}
\end{figure}
Figure~\ref{fig:circle_swoop_traces} shows a trace of a swoop training trajectory. Note the clear influence of ground-effect on the $Z$-tracking error as the quadrotor dips towards the ground and settles into the straight line path. These are exactly the sort of effects we wish to learn. 

The complete dataset consists of approximately 2 minutes of flight time, logged at 250 Hz. After filtering the finite-difference derivative estimates, we further sub-sampled the data to avoid aliasing effects and extracted training and validation datasets of  1000 and 3000 $(x,u,\dot{x})$ tuples respectively. 

\medskip

\noindent{\bf Models Learned}: To investigate the hypothesized trend where the control-theoretic regularization from the CCM approach dominates for small supervised training dataset sizes, we learned R-R and CCM-R models at $N=150$ and $N=1000$. For the CCM-R models, the full training constraint set $X_c$ comprised of the state samples from the demonstration tuples, plus up to 1600 additional state samples from the collected trajectories, and 500 independently sampled Halton points within the state-space. In total, for both the $N=150$ and $N=1000$ models, a total of 3100 constraint points were used. 

The dimensionality of the feature space for $\hat{f},\{\hat{b}_j\}_{j=1}^{2},$ and $W$ were held the same as from the PVTOL simulations (corresponding to $\alpha, \beta_j \in \reals^{576}$, and $\tilde{\theta}_{ij} \in \reals^{72}$). The constant $\mu_f$ was held fixed at $10^{-3}$ for both R-R and CCM-R models and both $N$, and $\mu_b$ was set to $10^{-2}$. The need for a higher regularization constant for $\beta_j$ stemmed from the high noise content in the $\dot{\phi}_c$ and $\dot{\omega}_x$ signals. The constant $\mu_w$ was set to $10^{-3}$ for $N=150$ and $10^{-4}$ for $N=1000$, as in the PVTOL simulations. Figure~\ref{fig:hard_train_curves} gives the training curve plots for the $N=150$ and $N=1000$ CCM-R models, and Table~\ref{tab:model_H_reg} summarizes the final measures of performance. The contraction rate $\lambda$ was again set to $0.1$ and all other tolerances from the simulations were held the same. CCM-R models converged in $5$ iterations for $N=150$ and $4$ iterations for $N=1000$. 

\begin{figure}[h]
\centering
\begin{subfigure}[t]{0.7\textwidth}
	\includegraphics[width=1\textwidth]{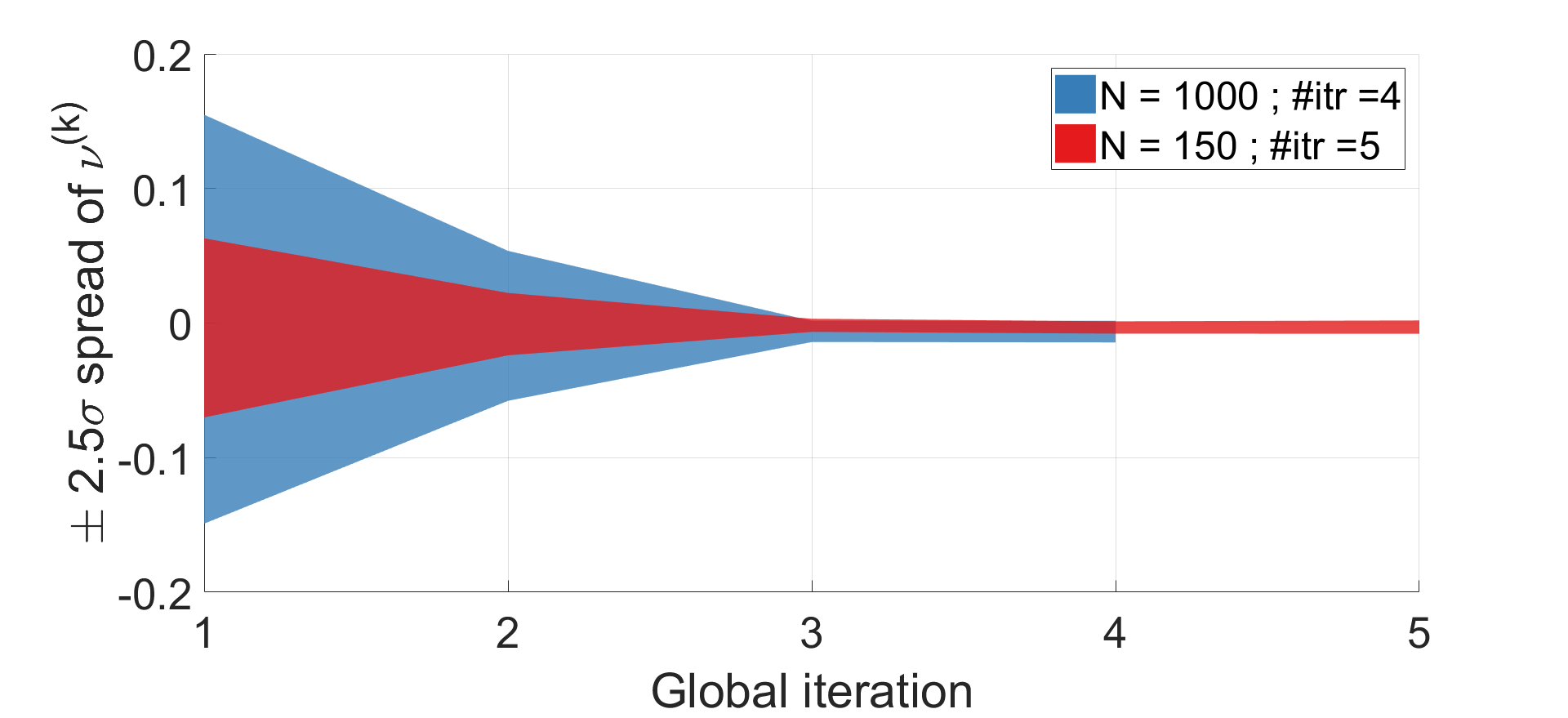}
	\caption{$\pm 2.5 \sigma$ spread of the constraint violation vector $\nu$ over \emph{training set} as a function of global iteration. $\#$\texttt{itr} denotes the total number of global iterations. }
	\label{fig:sim_curve_viol}
\end{subfigure} \qquad 
\begin{subfigure}[t]{0.7\textwidth}
	\includegraphics[width=1\textwidth]{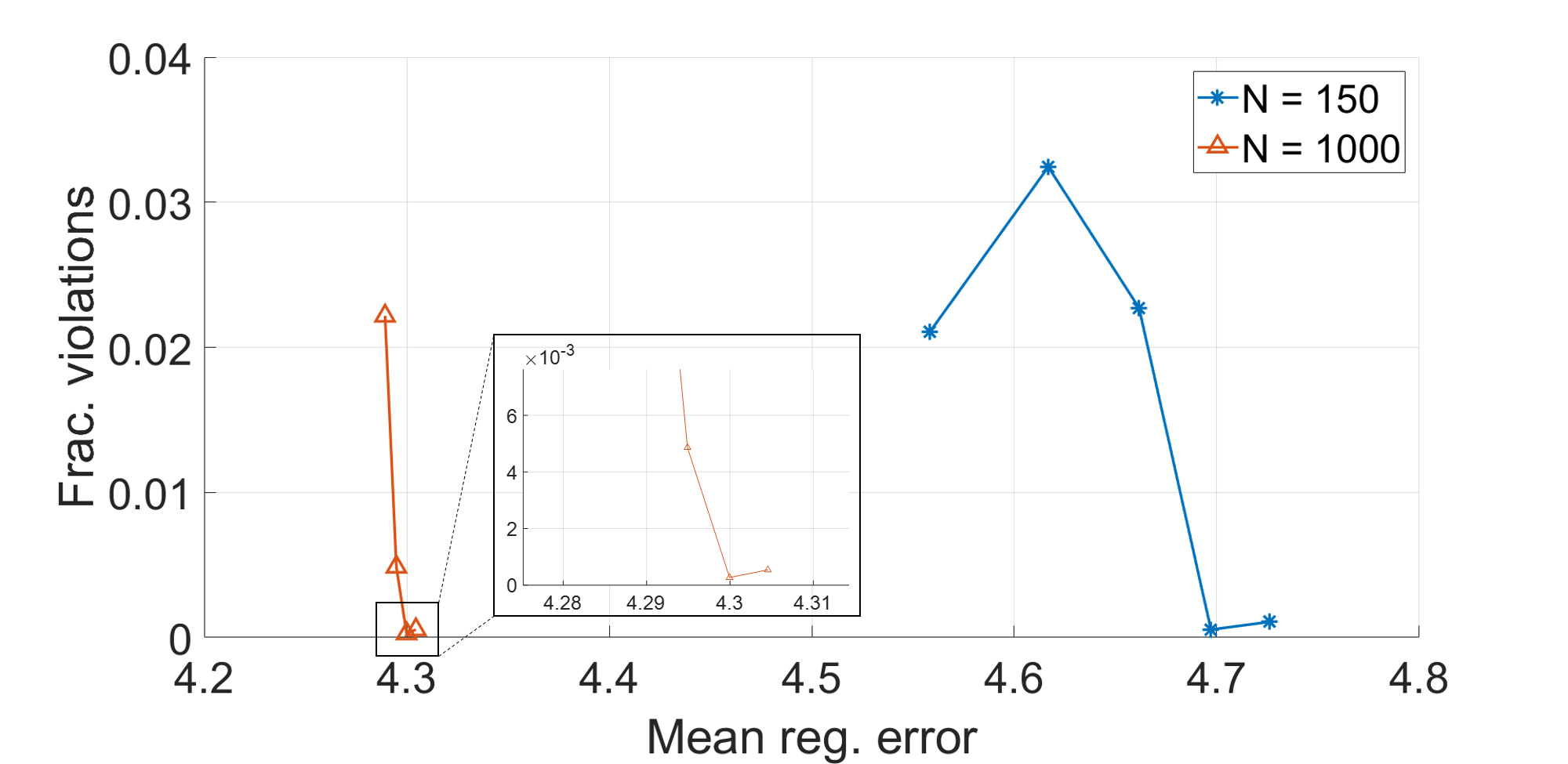}
	\caption{Evolution of mean regression error norm and fraction of violations over \emph{validation} set with global iteration number. The markers delineate the iterations. The curves proceed from left to right.}	
	\label{fig:hard_curve_reg}
\end{subfigure}
	\caption{Testbed data training curves for both CCM-R models.}
	\label{fig:hard_train_curves} 
\end{figure}

\begin{table}[H]\centering
\begin{tabular}{@{}lcccccc@{}}
\toprule
& \multicolumn{2}{c}{\textbf{R-R}} & \phantom{a} &  \multicolumn{3}{c}{\textbf{CCM-R}} \\
\cmidrule{2-3} \cmidrule{5-7}
$N$ & Train err.  & Val err. && Train err. & Val err. & Frac. viol. \\
\midrule
150 & 4.077 & 4.549 && 0.399 & 4.726 & 0.0013 (0.0011) \\
1000 & 4.149 & 4.289 && 0.167 & 4.305 & 0.0048 (0.0005) \\
\bottomrule
\end{tabular}
\caption{Comparison of average (over $N$ tuples) training and validation (over 3700 tuples) regression error norms for all 3 models. Also shown are the fraction of violations ($\nu>0$) on the training (validation) sets for the CCM-R models. The termination threshold for CCM-R training was $\nu < \varepsilon = 0.01$ for all points in the training constraint set. }
\label{tab:model_H_reg}
\end{table}

From Table~\ref{tab:model_H_reg}, we note that the final regression validation errors for R-R and CCM-R are again close together, as in the simulation case. However, the training curves in Figure~\ref{fig:hard_train_curves} appear to go in the opposite direction, i.e., validation regression error for CCM-R grows with iterations, starting from around 4.558 (almost identical to the R-R model), and growing to a final value of 4.726 as the number of violations converges to 0. A potential reason could be that the noise in the demonstration tuples (i.e., from $\dot{x}$) forces more drastic updates in the parameters in order to find a compatible dual metric $W$. Nevertheless, as will be seen in the evaluation results, this small sacrifice in validation error performance works in the CCM-R model's favor. 

\subsection{Out-of-Plane Control for Enabling Evaluation}\label{sec:quad_oop}

The out-of-plane errors in the training data were quite small (since the tracking controller was a full 3D state feedback controller). The learned planar dynamics will be used to generate a planar desired trajectory and LQR feedback tracking controller. To ensure that the quadrotor remains within the plane however, we need to separately close the loops on inertial $X$ and yaw motion. 

The nominal (ignoring disturbances) equation for $p_x$ is given by:
\[
    \ddot{p}_x = -\tau \sin(\theta), 
\]
which is completely decoupled from the planar dynamics. In order to regulate $p_x$ to an arbitrary constant $\bar{p}_x$ (without loss of generality, assume $\bar{p}_x = 0$), given a desired $\tau>0$ computed using the planar control loop, we choose a desired inertial acceleration $\ddot{p}_{x,des} := -k_x p_x - k_v \dot{p}_x$, where $k_x, k_v >0$ are control gains. Inverting the dynamics above, we obtain a desired pitch command $\theta_c$: 
\[
    \sin(\theta_c) = -\dfrac{\ddot{p}_{x,des}}{\tau} = \dfrac{k_x p_x + k_v \dot{p}_x}{\tau}. 
\]
From the above equation, we deduce a desired nominal pitch rate $\dot{\theta}_{des}$:
\[
    \dot{\theta}_{des} \cos(\theta_c) = \dfrac{\tau (k_x \dot{p}_x + k_v\ddot{p}_x) - \dot{\tau} (k_x p_x + k_v \dot{p}_x)}{\tau^2} \approx \dfrac{k_x \dot{p}_x}{\tau},
\]
where we neglect the $\dot{\tau}$ and $\ddot{p}_x$ terms. The actual pitch rate command $\dot{\theta}_c$ sent to the quadrotor is then given by:
\[
    \dot{\theta}_c = \dot{\theta}_{des} + k_{\theta} (\theta_c - \theta),
\]
which is a combination of the feedforward desired pitch rate plus a proportional feedback term on the error with respect to the desired pitch angle with gain $k_{\theta}>0$. Finally, to keep yaw $\psi$ at zero, we implemented a simple proportional controller to generate a desired yaw rate command $\dot{\psi}_c$:
\[
    \dot{\psi}_c = -k_{\psi} \psi,
\]
with gain $k_{\psi} > 0$. In the next section we provide plots of $p_x, \theta,$ and $\psi$ to check how well the planar assumption holds.

\subsection{Evaluation}

\noindent{\bf Test Trajectory Generation}: The evaluation task is similar to that used for simulations, except that instead of initializing the quadrotor at random initial configurations and asking it to stabilize to a hover (an arguably dangerous test), we designed a desired nominal trajectory in the $Y$-$Z$ plane consisting of segments of a continuous, smooth figure-eight, denoted as $(p_y^{\mathrm{ref}}(t), p_z^{\mathrm{ref}}(t))$. The nominal state and control trajectory for the quadrotor was then computed as the solution to the following trajectory optimization problem: 
\[
    (x^*(\cdot), u^*(\cdot)) = \argmin_{x(\cdot), u(\cdot)} \int_{0}^{T} \begin{bmatrix} p_y(t) - p_y^{\mathrm{ref}}(t) \\ p_z(t) - p_z^{\mathrm{ref}}(t) \end{bmatrix}^T Q \begin{bmatrix} p_y(t) - p_y^{\mathrm{ref}}(t)\\ p_z(t) - p_z^{\mathrm{ref}}(t) \end{bmatrix} + \begin{bmatrix} \tau(t)-g \\ \dot{\phi}_c(t) \end{bmatrix}^T R \begin{bmatrix} \tau(t)-g \\ \dot{\phi}_c(t) \end{bmatrix} dt,
\]
where $g$ is the gravitational acceleration, $T = 10$ s, and $Q, R$ are cost weighting matrices in $\Sjpp_2$. To avoid the initial jump in desired velocity for a quadrotor starting from rest, we designed the trajectory in three segments (see Figure~\ref{fig:figure8_hardware}). The quadrotor accelerates along half of the figure-eight path to full speed (point A to B in Figure~\ref{fig:figure8_hardware}) in the first segment. This is followed by a full figure-eight loop (back to B) in the second segment. Finally the quadrotor decelerates along the second half of the figure-eight back to zero velocity during the third segment. Each segment was designed to overlap with the desired figure-eight shape, and the desired speed was modulated along the path using a smooth time-varying phase function for the first and third segments. These two segments along with the speed along the trajectory are depicted in Figure~\ref{fig:figure8_hardware}.
\begin{figure}[h]
    \centering
    \includegraphics[width=0.75\textwidth]{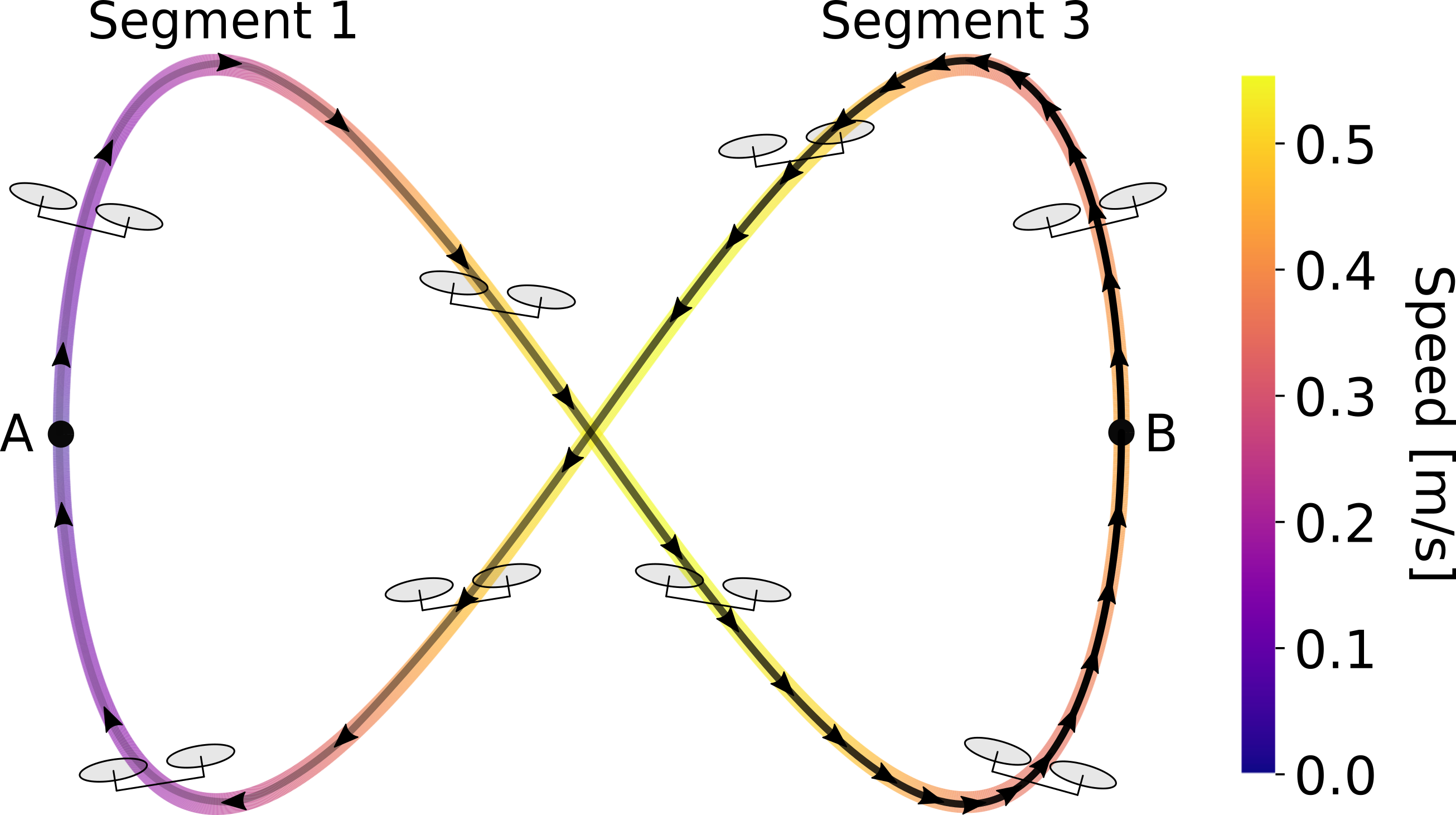}
	\caption{Illustration of the desired $Y$-$Z$ plane trajectory to be flown for evaluation of the learned models. Shown here are the first (accelerating from A to B) and third (decelerating from B to A) segments of the trajectory with the indicated speed profile. The middle segment (not shown) is a complete figure eight maneuver that overlaps with the above path. The actual reference state/control trajectories for each model are computed using trajectory optimization where the cost to be minimized is a combination of control effort and the deviation from the desired figure-eight path above. All computed trajectories, projected onto the $Y$-$Z$ plane overlap with the desired figure-eight maneuver.}
	\label{fig:figure8_hardware}
\end{figure}

For each model tested, three trajectory optimization problems were solved corresponding to the three segments, each with the same time-span of 10 s. The middle segment therefore represents the most challenging and aggressive part of the trajectory, executed following the acceleration segment. The planar tracking controller was set as the steady-state LQR controller computed from the TV-LQR solution (to avoid having to additionally compute and store the true time-varying gain solution from the backward Riccati differential equation).

\medskip

\noindent{\bf Results}: We present tracking results\footnote{Videos of the flight experiments can be found at \url{https://youtu.be/SK1tsYrXXUY}.} for all 4 models: (i) CCM-R $N=1000$, (ii) R-R $N=1000$, (iii) CCM-R $N=150$, and (iv) R-R $N=150$. Figure~\ref{fig:hard_results} plots tracking errors for the $N=150$ case and Table~\ref{tab:all_err} summarizes the tracking performance for all models in terms of RMS and maximum translational tracking errors. The quadrotor is unable to track the R-R $N=150$ model generated trajectory and quickly becomes unstable as it accelerates into the middle segment of the trajectory. A human pilot needed to take control just as the quadrotor crashed into the ground. At the point of the takeover, the net Y tracking error was approximately 1.5 m. 

In contrast, while the tracking performance for the CCM-R model is rather far from what could be achieved with an accurate model, all errors remain bounded. There is a point where the quadrotor grazes the ground as it passes through the fastest part of the middle segment, traveling downwards. The margin of error (distance between the lowest point on the trajectory and the ground) is 15 cm, which is a challenging tracking constraint to meet for a quadrotor operating using a model trained with just 150 samples of supervision. Despite this, the quadrotor recovers and successfully completes the remaining portion of the trajectory. Figure~\ref{fig:150_overlays} shows a time-lapse of the quadrotors during this middle segment of the maneuver. 

From Table~\ref{tab:all_err}, as expected for the $N=1000$ case, the tracking numbers for CCM-R and R-R are on par, with remaining differences at the noise level. For the $N=150$ case, we only present numerics up until 14 s, which is when the R-R case crashes. For this range, CCM-R significantly outperforms R-R, let alone the fact that the quadrotor operating with the CCM-R model manages to complete the entire 30 s trajectory while maintaining stability. The results confirm the trend observed within simulations that, at low sample regimes, the stabilizability constraints enforced during the CCM-R model learning process have a dramatic context-driven regularization effect.

\begin{figure}[H]
    \centering%
    \includegraphics[width=\textwidth]{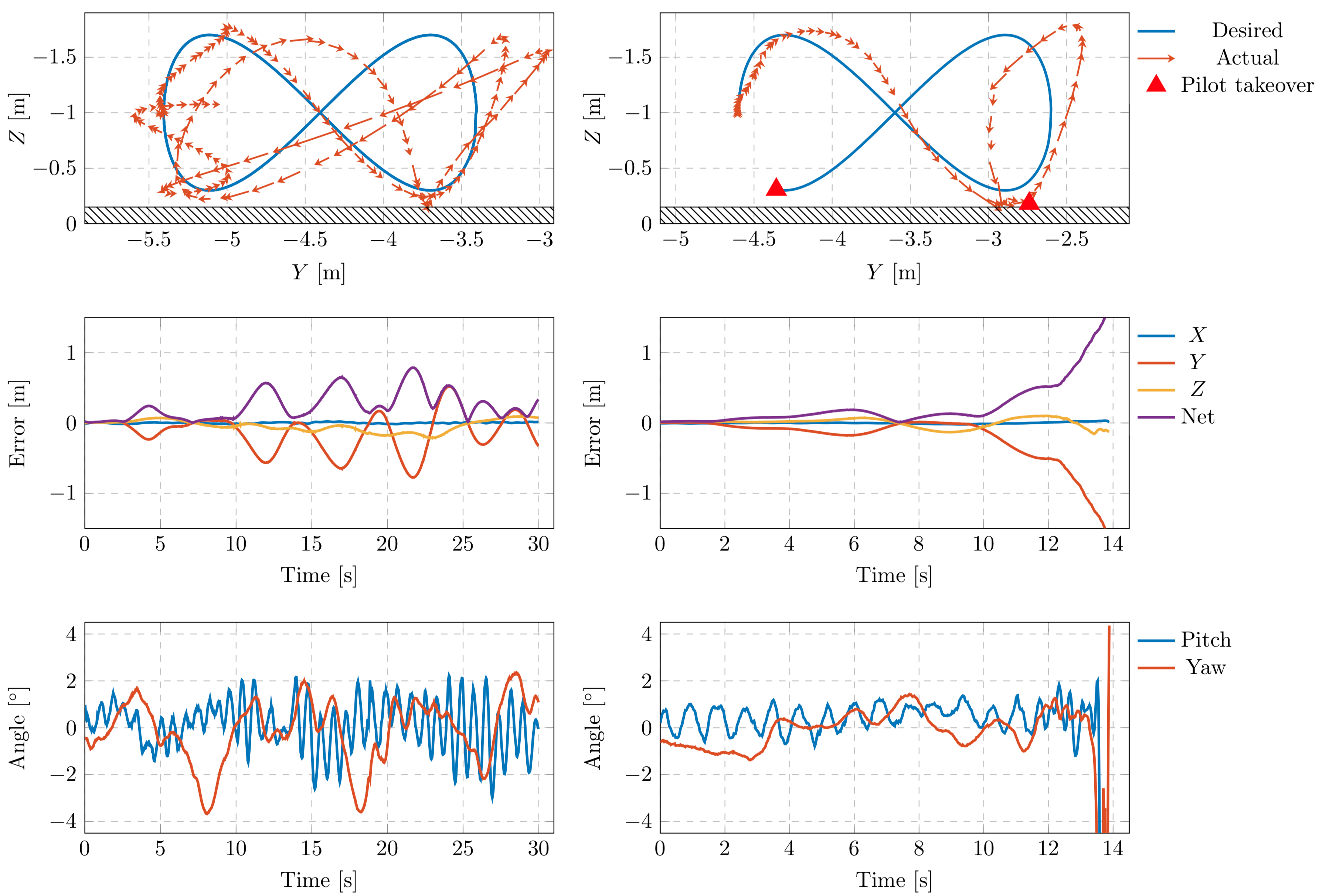}%
    \caption{Visualization of tracking results for the learned models CCM-R $N=150$ (left column) and R-R $N=150$ (right column). The top row shows the desired and actual trajectory traces in the inertial $Y$-$Z$ plane, while the middle row shows the error over time in the inertial $X$, $Y$, and $Z$ coordinates. The quadrotor is unable to track the R-R $N=150$ model generated trajectory and quickly becomes unstable as it accelerates into the middle segment of the trajectory. A human pilot needed to take control just as the quadrotor crashed into the ground. At the point of the takeover, the net Y tracking error was 1.5 m. Tracking the CCM-R $N=150$ model generated trajectory, the quadrotor achieves bounded tracking errors and successfully completes the trajectory. The bottom row shows the pitch and yaw angles over time, which remain close to zero, thereby ensuring our planar motion assumption is valid.}%
    \label{fig:hard_results}
\end{figure}



\begin{table}[H]\centering
\begin{tabular}{@{}lccccccccccc@{}}
\toprule
    & \multicolumn{2}{c}{\textbf{CCM-R $\bm{N=1000}$}}  & \phantom{\,} 
    & \multicolumn{2}{c}{\textbf{R-R $\bm{N=1000}$}}    & \phantom{\,} 
    & \multicolumn{2}{c}{\textbf{CCM-R $\bm{N=150}$}}   & \phantom{\,} 
    & \multicolumn{2}{c}{\textbf{R-R $\bm{N=150}$}} \\
            \cmidrule{2-3}     \cmidrule{5-6}      \cmidrule{8-9}      \cmidrule{11-12}
Error   & RMS   & Max.     && RMS   & Max.     && RMS   & Max.     && RMS   & Max.  \\
\midrule
$Y$         & 0.150 & 0.445    && 0.126 & 0.290    && 0.208 & 0.567    && 0.423 & 1.618 \\
$Z$         & 0.066 & 0.167    && 0.076 & 0.196    && 0.047 & 0.090    && 0.068 & 0.157 \\
Net         & 0.165 & 0.446    && 0.148 & 0.308    && 0.213 & 0.570    && 0.429 & 1.623 \\
\bottomrule
\end{tabular}
\caption{Numerical tracking results for all learned models. These include both root mean square (RMS) and maximum error values (in meters) over the entire trajectory. At $N=1000$ training points, both the CCM-R and R-R models yield similar tracking performance with remaining differences at the noise level. The tracking numbers for $N=150$ are presented up until 14 s, which is when the crash occured with the R-R model. CCM-R clearly outperforms R-R during those first 14 s, let alone the fact that the quadrotor operating with the CCM-R model successfully completes the entire 30 s trajectory while maintaining bounded errors.}
\label{tab:all_err}
\end{table}

\section{Conclusions and Future Work}

We presented a framework for learning \emph{controlled} dynamics from demonstrations for the purpose of trajectory optimization and control for continuous robotic tasks. By leveraging tools from nonlinear control theory, chiefly, contraction theory, we introduced the concept of learning \emph{stabilizable} dynamics, a notion which guarantees the existence of feedback controllers for the learned dynamics model that ensures trajectory trackability. 
Borrowing tools from  Reproducing Kernel Hilbert Spaces and convex optimization, we proposed a bi-convex semi-supervised algorithm for learning the dynamics and provided a substantial numerical study of its performance and comparison with traditional regression techniques. In particular, we validated the algorithm within simulations for a planar quadrotor system, and on a full quadrotor hardware testbed with partially closed loops to emulate planar quadrotor dynamics. The results lend credence to the hypothesis that enforcing stabilizability constraints during the learning process can have a dramatic regularization effect upon the learned dynamics in a manner that is tailored to the downstream task of trajectory generation and feedback control. This effect is most apparent when learning from smaller supervised training datasets, where we showed, both in simulation and on hardware, the quadrotor losing control and crashing when using a model learned with traditional ridge-regression. In contrast, the quadrotor is able to maintain bounded tracking performance when using the stabilizability regularized model in the low sample learning regime, and is on par with the ridge regularized model in the large sample regime. 

\subsection{Challenges and Extensions}

There are several exciting future directions that we would like to pursue, categorized below. 

\medskip

\noindent {\bf Multi-Step Error}: An immediate theoretically sound extension of the work to incorporate \emph{multi-step} regression error would be to leverage the linearity in parameters of the estimated model and integrate the nonlinear features and the control signal in time. This would result in an inner-product between time-varying features and the parameters, thereby retaining linearity in parameters, with the supervisory signal now being state samples along the trajectory, as opposed to the time-derivative of the state. A similar transformation is central also to composite adaptive control for robotic systems~\citep{SlotineLi1989,WensingSlotine2018}. While such a transformation would additionally require the full control time signal for each sampled trajectory (as opposed to state-control samples), one eliminates the need for numerical estimation of the state time-derivative along the trajectory, thereby improving the signal-to-noise ratio.

\medskip

\noindent {\bf Leveraging Model Priors}: The problem studied in this work concerned learning the full dynamics model from scratch. However, one can often leverage physics-based modeling to form a baseline, reducing the role of learning to estimate the residual dynamics. Instead of adopting an additive error formulation as is common in the literature, a variation of contraction theory -- partial contraction~\citep{WangSlotine2005} -- may be utilized to \emph{smoothly interpolate} between the prior model and the model learned from data using an auxiliary dynamical system, termed the ``virtual dynamics." This system has the special property in that the prior model and the learned model are both particular solutions to the virtual dynamics. Using similar stabilizability constraints, one can construct a model and feedback controller that forces the solution of the learned dynamics model (representing the true dynamics) to converge exponentially to the solution of the prior model, which can be designed to satisfy desirable performance criteria. 

\medskip

\noindent {\bf Notion of Stabilizability}: The form of control-theoretic regularization employed in this work was founded upon conditions for exponential stabilizability of smooth systems via state feedback. In order to broaden the applicability of incorporating control-theoretic notions within learning, it is of interest to investigate not only other, more general forms of stabilizability (e.g., boundedness), but also extend the analysis to hybrid mechanical systems. This would enable tackling more challenging problems within the manipulation and locomotion domains where prior models can be severely inaccurate due to the complexity of modeling contact physics. 

\medskip

\noindent {\bf Computational and Experimental}: By far the largest hurdle in being able to extend our algorithm to higher-dimensional systems (e.g., visual and contact domains) lies in addressing the LMI constraints. There are three key approaches that may be used to address this challenge. The first involves leveraging distributed optimization techniques such as ADMM~\citep{BoydParikhEtAl2011} to distribute and parallelize the constraints (which are encoded within the cost in sum-separable form). A second approach is to collapse the pointwise LMI constraints into a single LMI that is the weighted average of all the $\Fl$ matrices, where the weighting scales with the maximum eigenvalue of each of the $\Fl$ matrices. This is a non-convex constraint that can be solved using alternation by fixing the weights from the previous iteration. Indeed, this is equivalent to re-writing all the LMI constraints in Lagrangian form and performing primal-dual descent on the complementarity condition. While the above methods can allow scaling in terms of \emph{number of constraint points}, in order to scale with respect to problem dimensionality, one must resort to first-order unconstrained gradient descent methods, where non-degeneracy in the contraction metric can be achieved by parameterizing the dual metric $W$ in Cholesky form $L L^T$. This opens the possibility of using more expressive function approximators such as deep-neural-networks, albeit, at the expense of a fully non-convex formulation.

\medskip

We believe that the dynamics learning framework presented in this work provides compelling motivation from both a stability and data efficiency standpoint, for incorporating control-theoretic notions within learning as a means of context-driven hypothesis pruning.


\bibliographystyle{SageH}
\newcommand{\noopsort}[1]{} \newcommand{\printfirst}[2]{#1}
  \newcommand{\singleletter}[1]{#1} \newcommand{\switchargs}[2]{#2#1}

\newpage
\appendix
\section*{Appendix}
\section{Solution Parameters for PVTOL}
\label{app:prob_params}

Notice that the true input matrix for the PVTOL system satisfies Assumption~\ref{ass:B_simp}. Furthermore, it is a constant matrix. Thus, the feature mapping $\Phi_b$ is therefore just a constant matrix with the necessary sparsity structure.

The feature matrix for $f$ was generated using the random matrix feature approximation to the Gaussian separable matrix-valued kernel with $\sigma = 6$ and $s = 8n = 48$ sampled Gaussian directions, yielding a feature mapping matrix $\Phi_f$ with $d_f = 576$ (96 features for each component of $f$). The scalar-valued reproducing kernels for the entries of $W$ were taken to be the Gaussian kernel with $\sigma = 15$. To satisfy condition~\eqref{killing_A}, the kernel for $w_{ij}, \ (i,j) \in \{1,\ldots,(n-m)\}$ was only a function of the first $n-m$ components of $x$. A total of $s = 36$ Gaussian samples were taken to yield feature vectors $\phi_w$ and $\hat{\phi}_w$ of dimension $d_w = 72$.  Furthermore, by symmetry of $W(x)$ only $n(n+1)/2$ functions were actually parameterized. Thus, the learning problem in~\eqref{learn_finite} comprised of $d_f + d_w n(n+1)/2 + 4 = 508$ parameters for the functions, plus the extra scalar constants $\lambda,\wl,\wu$.

The learning parameters used were: model N-R (all $N$): $\mu_f = 0, \mu_b = 10^{-6}$, R-R (all $N$): $\mu_f = 10^{-4}, \mu_b = 10^{-6}$, CCM-R (all $N$): $\mu_f = 10^{-3}, \mu_b = 10^{-6}$; $N \in \{100,250,500\}: \mu_w = 10^{-3}$, $N = 1000: \mu_w = 10^{-4}$. Tolerance parameters: constraints: $\{\epsilon_{\lambda},\delta_{\wl}, \epsilon_{\wl}\} = \{0.1,0.1,0.1\}$; discard tolerance $\delta = 0.05$. Note that a small penalization on $\mu_b$ was necessary for all models due to the fact that feature matrix $\Phi_b$ is rank deficient.

\section{CCM Controller Synthesis} \label{ccm_appendix}

Let $\Gamma(p,q)$ be the set of smooth curves $c:[0,1] \rightarrow \X$ satisfying $c(0) = p, c(1) = q$, and define $\delta_c(s) := \partial c(s)/\partial s$. At each time $t$, given the nominal state/control pair $(x^*,u^*)$ and current actual state $x$:
\begin{leftbox}
\begin{enumerate}
    \item Compute a curve $\gamma \in \Gamma(x^*,x)$ defined by:
\begin{equation}
    \gamma \in \argmin_{c\in \Gamma(x^*, x)} \int_{0}^{1} \delta_c(s)^T M(c(s)) \delta_c (s) ds,
\end{equation}
    and let $\mathcal{E}$ denote the minimal value.
    \item Define: 
    \begin{equation}
    \begin{split}
        \mathcal{E}_d(k):= &2 \delta_{\gamma}(1)^T M(x) (f(x) + B(x) (u^* + k)) \\
        &- 2\delta_{\gamma}(0)^T M(x^*) (f(x^*) + B(x^*)u^*).
    \end{split}
    \label{ccm_ineq}
    \end{equation}
    \item Choose $k(x^*,x)$ to be any element of the set:
    \begin{equation}
        \mathcal{K} := \left\{ k : \mathcal{E}_d(k) \leq -2\lambda \mathcal{E} \right\}.
    \label{ccm_cntrl}
    \end{equation}
\end{enumerate}
\end{leftbox}

By existence of a CCM $M(x)$ (equivalently, its dual $W(x)$), the set $\mathcal{K}$ is always non-empty. The resulting $k(x^*,x)$ then ensures that the solution $x(t)$ indeed converges towards $x^*(t)$ exponentially~\citep{SinghMajumdarEtAl2017}.

From an implementation perspective, note that having obtained the curve $\gamma$, constraint~\eqref{ccm_cntrl} is simply a linear inequality in $k$. Thus, one can analytically compute a feasible feedback control value, e.g., by searching for the smallest (in any norm) element of $\mathcal{K}$; for additional details, we refer the reader to~\citep{ManchesterSlotine2017,SinghMajumdarEtAl2017}. This controller was not used for the CCM-R experiments in this paper however since the contraction conditions are only enforced at the discrete sampled constraint set, as opposed to over the continuous state space $\X$.

\end{document}